\documentclass{article}
\usepackage{lmodern}

\usepackage{silence}
\usepackage{xurl} 

\usepackage{amsmath,amsthm}
\usepackage{amssymb}
\usepackage{bm}

\usepackage{hyperref}
\usepackage{enumitem}
\usepackage{cleveref}
\usepackage[labelfont=bf]{caption}
\usepackage{subcaption}

\usepackage{float}
\usepackage{algorithm}
\usepackage{algpseudocode}

\usepackage{nicematrix}
\usepackage{geometry}
\usepackage{array}
\usepackage{placeins}

\usepackage[backend=biber,doi,url=false]{biblatex}
\bibliography{references.bib}

\crefname{problem}{problem}{problems}
\Crefname{problem}{Problem}{Problems}
\creflabelformat{problem}{#2(\textup{#1}#3)}

\newtheorem{theorem}{Theorem}[subsection]

\newtheorem{proposition}[theorem]{Proposition}
\newtheorem{lemma}[theorem]{Lemma}
\newtheorem{definition}[theorem]{Definition}

\newtheorem{example}[theorem]{Example}

\newtheorem{remark}[theorem]{Remark}
\newtheorem{hypothesis}[theorem]{Hypothesis}
\newtheorem{heuristic}[theorem]{Heuristic}

    \newcommand{\intpart}[1]{\lfloor #1 \rfloor}
    \newcommand{\fracpart}[1]{\{ #1 \}}
    
    \newcommand{\unif}{\mathcal{U}}
    \newcommand{\normal}{\mathcal{N}}
    
    \newcommand{\N}{\mathbb{N}}
    \newcommand{\Q}{\mathbb{Q}}
    \newcommand{\R}{\mathbb{R}}
    \newcommand{\Z}{\mathbb{Z}}
    \newcommand{\C}{\mathbb{C}}
    \newcommand{\T}{\mathbb{T}}

    \newcommand{\cont}{\mathcal{C}}
    \newcommand{\fov}{[-\frac{1}{2}, \frac{1}{2})}
    
    \newcommand{\fcoef}[1]{e^{-2 \pi i #1}}
    \newcommand{\ifcoef}[1]{e^{2 \pi i #1}}
    \newcommand{\ffcoef}[2]{e^{-\frac{2 \pi i #1}{#2}}}

    \DeclareMathOperator{\supp}{supp}

    \newcommand{\st}{\text{ $|$ }}
    
\begin{document}
    
    \title{\textbf{Optimization of gridding algorithms for FFT by vector optimization}}
    
    \author{Federico Achini$^{1, *}$, Paola Causin$^1$, Sara Vanini$^2$, Ke Chen$^3$ and Simone Scacchi$^1$ \\
    \small
    $^1$\textit{Department of Mathematics, Università degli Studi di Milano, Milano, Italy}, \\
    \small
    $^2$\textit{R\&D Department, Revenio Oyj - Centervue, Padova, Italy}, \\
    \small
    $^3$\textit{Department of Mathematics, University of Strathclyde, Glasgow, UK} \\
    \footnotesize
    $^*$federico.achini@unimi.it
    }
    \date{}
    
    \maketitle
    
    \begin{abstract}
    The Fast Fourier Transform (FFT) requires the input signal to be sampled
    uniformly, a condition that is frequently violated in practice, for instance
    in MRI, radio interferometry, and computed tomography. The standard remedy,
    known as gridding, convolves the non-uniformly sampled data with a compactly
    supported kernel, resamples it onto a uniform grid where the FFT can be
    applied, and corrects the result with a scalar function. The resulting
    three-step procedure is known as the adjoint non-uniform FFT (NUFFT). Its
    accuracy depends critically on the choice of the gridding kernel and
    correcting function, which existing methods optimize either heuristically or
    by minimizing a worst-case error bound that disregards both the statistical
    distribution of the samples and the possibility that only part of the
    spectrum is of practical interest. In this work we remove both restrictions.
    Assuming that the non-uniform time samples are asymptotically equidistributed
    with respect to a known measure, we derive a new error bound for the adjoint
    NUFFT that depends explicitly on this measure, generalizing existing bounds.
    We then formulate the design of the gridding kernel as a vector optimization
    problem over the resulting error profile, prove the continuity of the
    associated error operator, and establish the existence of Pareto-optimal
    kernels, both without constraints and with respect to a prescribed target
    error profile encoding a region of interest in the spectrum. Numerical
    experiments show that the kernels obtained in this way substantially improve
    accuracy on prescribed regions of interest and under non-uniform sampling
    distributions, while matching or exceeding the computational performance of
    the state-of-the-art FINUFFT library.
    
    \null
    
    \end{abstract}
    
    \section{Introduction}
    \label{sec:intro}
    \subsection{Background and statement of the problem}
    The Fast Fourier Transform (FFT) is a fundamental algorithm that, since its introduction by Cooley and Tukey in 1965 \cite{cooley_algorithm_1965, Brigham_FFT}, has proven to be a cornerstone in numerous applications such as MRI \cite{ansorge_physics_2016, yang_mean_2014, Fessler_minmaxNUFFT}, interferometry \cite{ye_optimal_2019, schwab_optimal_1980, brouw_aperture_1975}, and CT \cite{osullivan_fast_1985, lewitt_reconstruction_1983}, to name just a few. The FFT enables the computation of the Discrete Fourier Transform (DFT) of a discrete signal of length $N$ in only $O(N \log N)$ time, in contrast to the $O(N^2)$ complexity required by straightforward matrix-vector multiplication \cite{cooley_algorithm_1965}.  
    
    An important assumption for the application of the FFT is that the signal be \emph{sampled uniformly}. However, in some common real-world scenarios this is rarely the case, due to device imperfections or specific features of the underlying physical system \cite{li_compensation_2023,hogg_synthesis_1969,gallagher_introduction_2008}. Directly applying the FFT on non-uniformly sampled signals yields inaccurate or noisy results \cite{Dutt_NUFFT, Fessler_minmaxNUFFT, yang_mean_2014}. To address this limitation, signals are first approximated on a uniform grid using a gridding technique \cite{Fessler_minmaxNUFFT, osullivan_fast_1985, jackson_selection_1991}, that is, the data are interpolated onto a regular grid via convolution with a predefined \emph{gridding kernel}. The gridding kernel is assumed to have a small support, to ensure that the convolution can be fast computed (see \Cref{ch:Definition of the gridding algorithm}). Finally, the spectrum thus obtained is corrected by multiplication with a correcting function to improve the reconstruction accuracy. This algorithm (gridding + FFT + correction) is commonly known as \emph{adjoint non-uniform Fast Fourier Transform (NUFFT)} \cite{Fessler_minmaxNUFFT}, or \emph{Type II} NUFFT \cite{Dutt_NUFFT}. 
    
    Clearly, the choice of the gridding kernel and the correcting function affects the accuracy of the NUFFT. In this work, we address the problem of finding the optimal gridding kernel and correcting function depending on the two following hypothesis, which shall be clarified in the subsequent sections of the paper.
    \begin{hypothesis}
        \label{hyp:measure condition on time samples - informal}
        We assume that the time samples are equistributed with respect to a given known measure.
    \end{hypothesis}
    \begin{hypothesis}
    \label{hyp:ROI of spectrum - informal}
        We assume that only \emph{a region} of the spectrum is of real interest, that is, we really need to be accurate only on a fraction of the frequencies in the field of view.
    \end{hypothesis}
    We shall discuss in \Cref{ch:Applicability of the equidistribution hypothesis} to which extent \Cref{hyp:measure condition on time samples - informal} and \Cref{hyp:ROI of spectrum - informal} can be assumed in real-world applications. 
    \begin{remark}
    The key point is that one can use \textit{a priori} information on the class of signals to achieve a higher accuracy (or, equivalently, a faster computation) than the classical one. This opens the door to application-adapted NUFFT algorithms.
    \end{remark}
    In this work, we shall focus only on the adjoint NUFFT. The analogous algorithm, known as forward NUFFT \cite{Fessler_minmaxNUFFT} (or Type I NUFFT \cite{Dutt_NUFFT}), in which the time samples of the signals are uniform, but the DFT frequencies are not, shall not be considered in this work. However, the authors believe that the ideas regarding \Cref{hyp:ROI of spectrum - informal} on the Fourier data of the signals can be applied also to the forward NUFFT case.
    
    \subsection{Related works}
    The selection of the optimal gridding kernel and correcting function for both the forward and the adjoint NUFFT the has been the subject of extensive research \cite{schwab_optimal_1980, jackson_selection_1991, schomberg_gridding_1995, Fessler_minmaxNUFFT}. A first step towards optimality was made by Nguyen and Liu in \cite{Nguyen_NUFFT}, where it is proved that, for a fixed correcting function $h$, one can find an optimal gridding kernel $C$ by a simple least square minimization argument. The same argument can be adapted both for the forward and the adjoint NUFFT. The kernel $C$ found in this way can be fast computed for any choice of the vector of non-uniform time samples $\bm{t} \in \R^N$. However, no optimal correcting function $h$ was given, leaving the problem of finding the optimal setting $(C, h)$.
    
    Subsequently, Fessler and Sutton were able to numerically solve the problem for the forward NUFFT \cite{Fessler_minmaxNUFFT}. Their idea is based on the computation of a sharp bound for the forward NUFFT approximation error which depends only on the frequency $\omega$ and the correcting function $h$. They hence numerically minimize the worst-case value of this bound (that is, the $L^\infty$ norm of the bound) and find the optimal $h$ in this way. A similar strategy was adopted in \cite{yang_mean_2014} using the $L^2$ norm of the error bound instead of the $L^\infty$ norm. It is important to notice that, in both the cases, no assumption on the class of signals, nor on the relative importance of each frequency of the spectrum was made.
    
    In \cite{Fessler_minmaxNUFFT} it is also stated that, although the argument used for optimizing the forward NUFFT cannot be directly applied to the adjoint NUFFT, it may be the case that the found optimal gridding kernel and correcting function could be useful also for this latter problem. The results presented in \cite{fessler_nufft-based_2007} seem to confirm this idea, and also introduce a minimization of the error bound based on a weighted square norm which can be seen as a first attempt to introduce \Cref{hyp:ROI of spectrum - informal} into the problem. However, no analysis or numerical experiment was made and, ultimately, only uniform weights were effectively used.
    
    On a parallel path, in the gridding community the first prolate spheroidal wave function (PSWF) of zeroth order \cite{slepian_prolate_1961, landau_prolate_1961, slepian_prolate_1978} has been frequently cited (without any formal proof) as an optimal gridding kernel, together with a specific correcting function, due to its unique property of being the bandlimited function with maximum energy concentration in the time domain~\cite{barnett_parallel_2019,jackson_selection_1991, schwab_optimal_1980, brouw_aperture_1975}. This property helps to suppress image aliasing and artifacts in the reconstruction. Since direct computation of the first PSWF is challenging for this reason, it is usually approximated by the Kaiser–Bessel (KB) function, which is more practical and efficient, or, more recently, by the exponential of semicircle (ES) function.
    
    In \cite{Fessler_minmaxNUFFT}, it is proved that the optimal kernel and the correcting function found via the worst-case minimization technique used in \cite{Fessler_minmaxNUFFT} have just slightly better performances than KB, thus somehow confirming the choice of the PSWF by the gridding community.
    
    The above global approach, which is based on the worst-case error, has been questioned in scenarios where the goal is to have a very high accuracy only on a specific portion of the spectrum \cite{cao_locally_1995}. In \cite{ye_optimal_2019}, an assumption similar to \Cref{hyp:ROI of spectrum - informal} (but stronger) is made: the authors propose an optimization procedure to determine the optimal gridding kernel for restricted and symmetric around $\omega = 0$ regions of the field of view, demonstrating improved performance compared to the PSWF. The signals are assumed to satisfy a certain condition of uniformity on the time samples, which shall be analysed in the following sections of this paper.
    
    Although not explicitly stated, the methods used in \cite{ye_optimal_2019} belong to the realm of vector optimization theory (VO). In this work, we vastly generalize them in view of \Cref{hyp:ROI of spectrum - informal}, which do not make any assumption on the shape or symmetry of the region of interest of the field of view.
    
    \subsection{Original contributions and structure of the paper}
    The main contributions of this paper are:
    \begin{enumerate}
        \item an analysis of the error of the adjoint NUFFT that generalizes \cite{ye_optimal_2019}, based on \Cref{hyp:measure condition on time samples - informal},
        \item a new definition of optimality for gridding kernels from the perspective of VO which leads to an application-adapted kind of NUFFT,
        \item extensive numerical experiments validating the theoretical results.
    \end{enumerate}
    
    \noindent The following sections are organized as follows. In \Cref{ch:Error bound for the gridding algorithm} we extend the adjoint NUFFT error analysis in \cite{ye_optimal_2019} by assuming (a formal version of) \Cref{hyp:measure condition on time samples - informal}, obtaining a more general upper bound for the adjoint NUFFT approximation error. In \Cref{ch:Applying vector optimization to ell(x)} we recall the main concepts of VO, define an operator $\Lambda$ that describes the adjoint NUFFT error, prove its continuity, discuss the existence of optimal solutions in the minimization of $\Lambda$, and show some additional properties. In \Cref{ch:Numerical results} we show the numerical results of several tests that we performed. Finally, in \Cref{ch:Conclusions}, we draw the conclusions.
    
    \subsection{Notation}
    The following notation will be used throughout the paper. Given a measure space $X$, we denote by $L^p(X)$ the $L^p$ space of all measurable \textit{complex} functions on $X$ with finite $p$-norm, and by $L^p_\R(X)$ the space of all measurable \textit{real} functions on $X$ with finite $p$-norm. If $X$ is a topological space, the space $\cont(X)$ is given by all the continuous functions $X \rightarrow \C$. Given an interval $I \subseteq \R$ and an integer $k \geq 0$, then $\cont^k(I)$ denotes the space of all $k$-differentiable functions $f: I \rightarrow \C$. 
    
    If $f \in L^1(\R)$, then $\widehat{f}$ denotes its Fourier transform, which in this paper is defined by
    \begin{equation*}
        \widehat{f}(\omega) = \int_{\R} f(t) \fcoef{\omega t} dt.
    \end{equation*}
    We denote by $\Re(z)$, $\Im(z)$, and $\overline{z}$ the real part, the imaginary part, and the complex conjugate of $z \in \C$ respectively. 
    
    \null
    
    Vectors in $\R^N$ or $\C^N$ are denoted by bold letters, such as $\bm{u}$. The elements of $\bm{u}$ are denoted by $u_n$, $n = 0, \dots, N-1$, and we also write $\bm{u} = (u_n)$ when the dimension of $\bm{u}$ is understood. An order relation defined in a vector space will be denoted by $\preceq$, while the symbol $\leq$ is reserved to the total ordering in $\R$. 
    
    Other symbols not listed here are explained in the paper.

    \section{Adjoint NUFFT error analysis}
    \label{ch:Error bound for the gridding algorithm}
    In this section we describe the theoretical aspects of the adjoint NUFFT (\Cref{ch:Definition of the gridding algorithm}), and find an upper bound for its error which depends on \Cref{hyp:measure condition on time samples - informal} (\Cref{ch:An upper bound for the gridding algorithms error}). The results presented here are a vast generalization of those given in \cite{ye_optimal_2019}.
    
    \subsection{Definition of the adjoint NUFFT}
    \label{ch:Definition of the gridding algorithm}
    Suppose we are given a discrete signal $\bm{u} = (u_n)\in\C^N$, $n = 0, \dots, N-1$, with $N \gg 1$ an integer. We wish to compute the following non-uniform DFT for a vector of non-uniform time samples $\bm{t} = (t_n) \in [0, N)^N$ and a fixed frequency $\omega \in \fov$:
    \begin{equation}
        \label{eq:NDFT}
        y(\omega) = \sum_{n < N}u_n \fcoef{\omega t_n}.
    \end{equation}
    Computing \cref{eq:NDFT} involves $N$ complex products and $N - 1$ additions: thus, it can be done in $O(N)$ operations. Usually, we have to compute this quantity for $N$ different frequencies $\omega_m$ uniformly distributed in $\fov$, with $m = 0, \dots, M-1$, and compute $y(\omega_m)$ for each of them. Thus, we get a total of $O(N^2)$ operations and, if $N$ is large, the computational cost becomes overwhelming. 
    
    The goal of the adjoint NUFFT is to fast compute an approximation of $y(\omega_m)$, for all $m = 0, \dots, N-1$, where "fast" means that the computational cost is $O(N\log(N))$ (the same as the FFT). The usual way \cite{Dutt_NUFFT,fessler_nufft-based_2007,barnett_parallel_2019} consists in approximating the signal $\bm{u}$ onto a uniform grid by convolving it with a pre-selected \emph{gridding kernel} $C$, then applying the FFT, and finally correcting the result by multiplying it with a \emph{correcting function} $h$. The resulting approximation $y^*(\omega)$ of $y(\omega)$ is
    \begin{equation}
        \label{eq:NFFT}
        y^*(\omega) = h(\omega)\sum_{k < \gamma N} u^*_k \fcoef{\omega k/\gamma},
    \end{equation}
    where $\gamma \geq 1$ is fixed, and
    \begin{equation}
        \label{eq:gridding}
        u^*_k = \sum_{n < N}u_n C(k - \gamma t_n \mod \gamma N).
    \end{equation}
    The vector $\bm{u}^* = (u_k^*)$ is the resampled signal. The parameter $\gamma$ is called \emph{oversampling factor}. Notice that $\bm{u}^*$ belongs to $\C^{\intpart{\gamma N}}$, where $\intpart{\gamma N}$ denotes the largest integer smaller than $\gamma N$. Thus, we do not exclude the possibility that the signal is interpolated on a finer grid. We assume that $C \in \cont^2(\R)$, with $\supp(C) \subseteq I_W = [-W, W]$ for some chosen parameter $W \geq 1$, and $h \in \cont(\fov)$. In the following, we shall write simply $C(k - \fracpart{\gamma t_n})$ instead of $C(k - \fracpart{\gamma t_n} \mod \gamma N)$, and we shall give as understood that the argument of $C$ is intended as modulo $\gamma N$.
    \begin{remark}
    \label{rmk:C is a function on the circle}
        The function $C$ could be interpreted as in $\cont^2(\T)$, where $\T$ is the 1-dimensional torus, under the identification $\T \cong [-\intpart{\gamma N}/2, \intpart{\gamma N}/2)$. However, we shall not use this definition here.
    \end{remark}
    
    The scalars $\gamma$ and $W$, and the functions $C$ and $h$, should be considered as parameters that can be used to calibrate the NUFFT, depending on our specific requirements. In general, smaller values of $W$ and $\gamma$ shall result in faster but less accurate performance, and vice versa. 
    
    \begin{remark}[Computational cost]
        Assume that $M = N$ for simplicity, and $W \leq N$. Then \cref{eq:gridding} can be calculated with only $O(W)$ operations, since most of the $C(k - \gamma t_n)$ terms are null ($C$ has support in $I_W$). Repeating the computation for all $k < \gamma N$, we get a total number of operations of $O(\gamma W N)$. Using the FFT to compute \cref{eq:NFFT}, we can compute the vector $(y^*(\omega_m)) \in \C^M$ in only $O(\gamma N\log (\gamma N))$) operations. Finally, the multiplication with $h$ requires $O(N)$ operations. Summing all the contributions, the total number of operations required by the adjoint NUFFT is $O(\gamma W N + \gamma N \log(\gamma N))$: a number much smaller than $O(N^2)$, provided that $N \gg 1$ and $W, \gamma = O(1)$.
    \end{remark}
    
    \subsection{Error bound}
    \label{ch:An upper bound for the gridding algorithms error}
    Let $\omega$ be a fixed frequency in $\fov$. The usual way to bound the NUFFT error is by computing the squared error between $y(\omega)$ and $y^*(\omega)$ \cite{fessler_nufft-based_2007, ye_optimal_2019}, which is
    \begin{equation}
        \label{eq:error}
        \begin{split}
            E(\bm{u}; \omega, h, C) & = |y(\bm{u}; \omega, h, C) - y^*(\bm{u}; \omega, h, C)|^2 
            \\
            & = \Big|\sum_{n < N}u_n \fcoef{\omega t_n} - h(\omega)\sum_{k < \gamma N}\sum_{n< N}u_n C(k  - \gamma t_n) \fcoef{\omega k/\gamma}\Big|^2 
            \\     
            & = \Big| \sum_{n < N}u_n \fcoef{\omega t_n} \cdot \Big(1 - h(\omega) \sum_{k < \gamma N}C(k  - \gamma t_n)\fcoef{\omega(k/\gamma - t_n)}\Big)\Big|^2 \\
            & \leq \sum_{n < N}|u_n|^2 \cdot \underbrace{\sum_{n < N}\Big|1 - h(\omega) \sum_{k < \gamma N}C(k  - \gamma t_n)\fcoef{\omega(k/\gamma  -  t_n)}\Big|^2}_{U(\omega, C, h)}. 
        \end{split}
    \end{equation}
    We see that the last row contains a term that depends only on the signal $\bm{u}$, and a term $U$ that depends only on $\omega$, $C$, and $h$. This latter term is the most interesting, since it constitutes an error bound which does not depend on $\bm{u}$.
    In general, one wants to keep it as small as possible. In particular, if we assume that the gridding kernel $C$ has been fixed, we can minimize the bound depending on the correcting function $h$. 
    
    Indeed, since $|z|^2 = z \overline{z}$ for any $z \in \C$, we have
    \begin{equation}
        \label{eq:bound_open}
        \begin{split}
            U(\omega, C, h) 
            & = N - h(\omega) \underbrace{\sum_{n < N}\sum_{k < \gamma N}C(k - \gamma t_n)\fcoef{\omega(k/\gamma -  t_n)}}_A - \overline{h}(\omega) \underbrace{\sum_{n < N} \sum_{k < \gamma N}\overline{C}(k - \gamma t_n)\ifcoef{\omega(k/\gamma -  t_n)}}_{\overline{A}}
            \\
            & \quad + |h(\omega)|^2 \underbrace{\sum_{n < N} \sum_{k_1, k_2 < \gamma N}C(k_1 - \gamma t_n)\overline{C}(k_2 - \gamma t_n)\fcoef{\omega(k_1 - k_2)/\gamma}}_{B}
            \\
            & = N - h(\omega) A - \overline{h}(\omega)\overline{A}+|h(\omega)|^2B \\
            & = N - \Re(h(\omega))(A + \overline{A}) - i\cdot\Im(h(\omega))(A - \overline{A}) + \Re(h)^2 B + \Im(h(\omega))^2 B \\
            & = N - 2 \Re(h(\omega)) \Re(A) + 2\Im(h(\omega))\Im(A) + \Re(h(\omega))^2 B + \Im(h(\omega))^2 B
        \end{split}
    \end{equation}
    The last row shows that the error bound $U$ is a \textit{real} polynomial of degree 2 in the \textit{real} variables $\Re(h(\omega))$ and $\Im(h(\omega))$, which we denote by $F(\Re(h(\omega)), \Im(h(\omega)))$, and which describes an upward elliptic paraboloid (notice that $B$ is real and positive by definition). Therefore, there exists a unique function $h^*$ such that $(\Re(h^*(\omega)), \Im(h^*(\omega))$ is the minimum point of $U$.
\begin{remark}
\label{rmk:C is not zero}
  We can always suppose that, for all $n = 0, \dots, N-1$ and for all $\omega \in \fov$, the quantity
  \begin{equation*}
      \sum_{k < \gamma N} C(k - \gamma t_n) \fcoef{\omega(k/\gamma - t_n)}
  \end{equation*}
  is non-zero. Otherwise, the gridding kernel could be used to interpolate the discrete signal at $t_n$.
\end{remark}
Given the above remark, we have $B > 0$. A simple analysis of the partial derivatives then reveals 
    \begin{equation*}
    (\Re(h^*(\omega)), \Im(h^*(\omega))) = \big(\frac{\Re(A)}{B}, -\frac{\Im(A)}{B}\big).
    \end{equation*}Therefore, we get
    \begin{equation}
    \label{eq:optimal h definition with A and B}
        h^*(\omega, C) = \frac{\overline{A}}{B},
    \end{equation}
    which realizes the minimum
    \begin{equation}
    \label{eq:error bound depending on t_n}
        U^*(\omega, C) = N - \frac{|A|^2}{B}.
    \end{equation}
    \begin{remark}
        The error bound in \eqref{eq:error bound depending on t_n} still depends on the time vector $\bm{t}$, making it unsuitable to find a general error bound independent of the sampling.
    \end{remark}
    In \cite{fessler_nufft-based_2007} the problem highlighted in the above remark was solved directly from \eqref{eq:error}, noticing that
    \begin{equation}
    \label{eq:fessler's bound}
      \sum_{n < N}\Big|1 - h(\omega) \sum_{k < \gamma N}C(k  - \gamma t_n)\fcoef{\omega(k/\gamma  -  t_n)}\Big|^2 \leq N \max_{t \in [0, N)}\Big|1 - h(\omega) \sum_{k < \gamma N}C(k  - \gamma t)\fcoef{\omega(k/\gamma  -  t)}\Big|^2.  
    \end{equation}
    The error bound found using this estimate corresponds to the worst-error case. The worst-error, although tight, is unlikely to be found in real-world applications. The "typical" error is generally better than \eqref{eq:fessler's bound} \cite{yang_mean_2014,Fessler_minmaxNUFFT,ye_optimal_2019}. 
    
    In this work, we shall follow a different idea: instead of bounding the error in full generality, we decided to find a bound which is reminiscent of the distribution of the time samples $\bm{t}$. This bound is based on \Cref{hyp:measure condition on time samples - informal}, which we now describe formally. First, notice that, since the argument of $C$ is defined modulo $\gamma N$,
    \begin{equation}
        \sum_{k < \gamma N}C(k - \gamma t_n)\fcoef{\omega(k/\gamma - t_n)} = \sum_{k < \gamma N}C(k - \fracpart{\gamma t_n})\fcoef{\omega(k - \fracpart{\gamma t_n})/\gamma},
    \end{equation}
    where $\fracpart{x}\in\fov$ denotes the signed fractional part of $x \in \R$. We then make the following assumption, whose validity in real-world scenarios shall be discussed in \Cref{ch:Applicability of the equidistribution hypothesis}.
    \begin{hypothesis}[Equidistribution of samples]
    \label{hyp:time samples are equidistributed - formal}
        For all $\gamma \geq 1$, there exists a limit measure $\nu = \nu(\gamma)$ defined on $\fov$ such that the time samples $t_n$ are equidistributed with respect to $\nu$, that is
        \begin{equation}
            \frac{1}{N}\sum_{n < N} \delta_{\fracpart{\gamma t_n}} \longrightarrow \nu, \qquad N \rightarrow \infty,
        \end{equation}
        where the limit is intended in the weak-$*$ topology. Therefore
        \begin{equation}
           \frac{1}{N}\sum_{n = 0}^{N-1} f(\fracpart{\gamma t_n}) \longrightarrow \int_{-1/2}^{1/2} f(\tau) d\nu(\tau), \qquad N \rightarrow \infty,\quad\forall f \in \cont\big(\fov \big).
        \end{equation}
        Furthermore, we assume that the limit measure $\nu$ is absolutely continuous with respect to the Lebesgue measure $\mu$ on $\fov$. 
    \end{hypothesis}
    
    The final absolute continuity assumption is added to avoid degenerate cases which are unlikely to be found in real-world scenarios. With the assumption of \Cref{hyp:time samples are equidistributed - formal}, we can prove the following theorem, which gives us an error bound independent of $\bm{t}$ and genuinely different from those in \cite{fessler_nufft-based_2007, yang_mean_2014}. However, before proving it, we need to find the correct generalization of \Cref{rmk:C is not zero} to the $\nu$-limit case.
\begin{remark}[Non-degeneracy of the gridding kernel]
\label{rmk:C is nu not zero}
We suppose that, for almost all $\omega \in \fov$, the quantity
\begin{equation}
    D_{C,\omega}(\tau)
    =
    \sum_{k < \gamma N}
    C(k-\tau)\fcoef{\omega(k-\tau)/\gamma}
\end{equation}
is nonzero for $\nu$-almost every $\tau\in\fov$. Equivalently, for $\nu$-almost every $\tau\in\fov$, there exists $k<\gamma N$ such that
\[
C(k-\tau)\neq0.
\]
Indeed, for fixed $\tau$, $D_{C,\omega}(\tau)$ is, up to the non-vanishing phase factor $\ifcoef{\omega\tau/\gamma}$, a finite trigonometric polynomial in $\omega/\gamma$. Hence, whenever at least one coefficient $C(k-\tau)$ is nonzero, it is not identically zero and therefore has only finitely many roots in $\fov$.
\end{remark}
    The assumption in the above remark is not very strong: basically, we are requiring the gridding kernel $C$ to be not-null with respect to the periodic extension of $\nu$ to all $\R$, which means that $\nu$ "can see" $C$. A trivial example of such a kernel is any smooth approximation of $C \equiv 1$ with support in $I_W$. More generally, any function that is real analytic in the interior of $I_W$ satisfies the condition.
    
    We are now ready to prove the main theorem of this section.
    \begin{theorem}[Adjoint NUFFT error bound]
        \label{thm:error bound}
        Fix $W, \gamma \geq 1$. Suppose that $N \gg 1$, and that \Cref{hyp:time samples are equidistributed - formal} holds for a limit measure $\nu$ defined over $\fov$. Fix a gridding kernel $C \in \cont^2(\R)$ with support in $I_W = [-W, W]$. The sharpest bound function for the error in the gridding algorithm is given by
        \begin{equation}
        \label{eq:ell bound definition}
            \ell_\nu(\omega; C) = N\left(
    1-
    \frac{
    \left|
    \sum_{s\in\Z}
    \widehat{\nu}_{-s}
    \widehat{C}(\omega/\gamma-s)
    \right|^2
    }{
    \sum_{s,s'\in\Z}
    \widehat{\nu}_{s'-s}
    \widehat{C}(\omega/\gamma-s)
    \overline{\widehat{C}(\omega/\gamma-s')}
    }
    \right)
    +o(N),
        \end{equation}
        which is realized if
        \begin{equation}
            \label{eq:optimal h}
            h_\nu^*(\omega; C) = \frac{
    \sum_{s\in\Z}
    \widehat{\nu}_s
    \overline{\widehat{C}(\omega/\gamma-s)}
    }{
    \sum_{s,s'\in\Z}
    \widehat{\nu}_{s'-s}
    \widehat{C}(\omega/\gamma-s)
    \overline{\widehat{C}(\omega/\gamma-s')}
    }
    +o(1).
    \end{equation}
\end{theorem}
\begin{proof}
    Consider the coefficient $A$ defined in \eqref{eq:bound_open}. Since $C$ has compact support $[-W, W]$, for each fixed $\fracpart{\gamma t_n}\in\fov$ only the terms $k\in\mathbb Z$ such that $|k - \fracpart{\gamma t_n}| \leq W$ give a non-zero contribution to
\begin{equation}
    C(k-\fracpart{\gamma t_n})\fcoef{\omega(k-\fracpart{\gamma t_n})/\gamma}.
\end{equation}
Consequently, the sum may be extended to all $k\in\mathbb Z$ without changing its value. Define
\begin{equation}
    \label{eq:D definition}
    D_{C, \omega}(\tau)
    =
    \sum_{k\in\mathbb Z}
    C(k-\tau)\fcoef{\omega(k-\tau)/\gamma},
    \qquad \tau\in\R.
\end{equation}
Then
\begin{equation}
\label{eq:A in D form}
    A
    =
    \sum_{n<N}
    D_{C, \omega}\left(\fracpart{\gamma t_n}\right).
\end{equation}
Furthermore, $D_{C, \omega}$ is $1$-periodic for every $\omega\in\fov$. Indeed,
\begin{equation}
\begin{split}
    D_{C, \omega}(\tau+1)
    &=
    \sum_{k\in\mathbb Z}
    C(k-\tau-1)
    \fcoef{\omega(k-\tau-1)/\gamma} =
    \sum_{j\in\mathbb Z}
    C(j-\tau)
    \fcoef{\omega(j-\tau)/\gamma} =D_{C, \omega}(\tau),
\end{split}
\end{equation}
where we have performed the change of variables $j=k-1$. Clearly, $D_{C, \omega}$ is twice continuously differentiable in any finite interval $I\subseteq\R$, since for all $\tau\in I$ only finitely many $k$ can contribute to the sum. Hence, $D_{C, \omega}\in\cont^2(\R)$. Together with the $1$-periodicity of $D_{C, \omega}$, this shows that $D_{C, \omega}$ can be regarded as an element of $\cont^2(\T)$. In particular, $D_{C, \omega}$ admits a Fourier series
\begin{equation}
    D_{C, \omega}(\tau)
    =
    \sum_{s\in\Z}
    \widehat{D_{C, \omega, s}}\ifcoef{s\tau}.
\end{equation}

We now compute the Fourier coefficients of $D_{C, \omega}$. By definition,
\begin{equation}
\begin{split}
    \widehat{D_{C, \omega, s}}
    &=
    \int_{-1/2}^{1/2}
    D_{C, \omega}(\tau)\fcoef{s\tau}\,d\tau
    =
    \sum_{k\in\Z}
    \int_{-1/2}^{1/2}
    C(k-\tau)
    \fcoef{\omega(k-\tau)/\gamma}
    \fcoef{s\tau}
    \,d\tau.
\end{split}
\end{equation}
The interchange between the sum and the integral is justified because,
for any interval $I \subseteq\R$ and for all $\tau \in I$, only finitely many terms are non-zero and the support of $C$ is compact.

For each $k\in\Z$, perform the change of variables $\xi=k-\tau$. Since $k\in\Z$, we have $\fcoef{sk}=1$, and therefore
\begin{equation}
\begin{split}
    \widehat{D}_{C, \omega, s}
    &=
    \sum_{k\in\Z}
    \int_{k-1/2}^{k+1/2}
    C(\xi)
    \fcoef{(\omega/\gamma-s)\xi}
    \,d\xi.
\end{split}
\end{equation}
The intervals
\begin{equation}
    [k-1/2,k+1/2],
    \qquad k\in\Z,
\end{equation}
form a partition of $\R$. Hence,
\begin{equation}
\begin{split}
    \widehat{D}_{C, \omega, s}
    &=
    \int_{\R}
    C(\xi)
    \fcoef{(\omega/\gamma-s)\xi}
    \,d\xi
    =
    \widehat{C}(\omega/\gamma-s).
\end{split}
\end{equation}
Consequently,
\begin{equation}
    \label{eq:D Fourier representation}
    D_{C, \omega}(\tau)
    =
    \sum_{s\in\Z}
    \widehat{C}(\omega/\gamma-s)\ifcoef{s\tau}.
\end{equation}

We can now use the equidistribution hypothesis. Since $D$ is continuous,
applying \Cref{hyp:time samples are equidistributed - formal} to \eqref{eq:A in D form} gives
\begin{equation}
\label{eq:expression for A}
    A
    =
    N\int_{-1/2}^{1/2}
    D_{C, \omega}(\tau)\,d\nu(\tau)
    +o(N).
\end{equation}
Substituting \eqref{eq:D Fourier representation} into
\eqref{eq:expression for A}, we obtain
\begin{equation}
\begin{split}
    A
    &=
    N
    \sum_{s\in\Z}
    \widehat{C}(\omega/\gamma-s)
    \int_{-1/2}^{1/2}
    \ifcoef{s\tau}\,d\nu(\tau)
    +o(N)
    =
    N
    \sum_{s\in\Z}
    \widehat{\nu}_{-s}
    \widehat{C}(\omega/\gamma-s)
    +o(N).
\end{split}
\end{equation}
Here we have used the definition of the Fourier coefficients
$\widehat{\nu}_s$ of the measure $\nu$.

We next consider the coefficient $B$. Using the same function $D_{C, \omega}$, we
can write
\begin{equation}
\begin{split}
    B
    &=
    \sum_{n<N}
    \left|
    \sum_{k\in\Z}
    C\left(k-\fracpart{\gamma t_n}\right)
    \fcoef{\omega(k-\fracpart{\gamma t_n})/\gamma}
    \right|^2
    =
    \sum_{n<N}
    \left|
    D_{C, \omega}\left(\fracpart{\gamma t_n}\right)
    \right|^2.
\end{split}
\end{equation}
Since $D_{C, \omega}$ is continuous, $|D_{C, \omega}|^2$ is also continuous, and therefore
\Cref{hyp:time samples are equidistributed - formal} yields
\begin{equation}
\label{eq:B integral}
    B
    =
    N
    \int_{-1/2}^{1/2}
    |D_{C, \omega}(\tau)|^2\,d\nu(\tau)
    +o(N).
\end{equation}
Notice that $B > 0$ because of \Cref{rmk:C is nu not zero}. Using \eqref{eq:D Fourier representation}, we obtain
\begin{equation}
\begin{split}
    |D_{C, \omega}(\tau)|^2
    &=
    \sum_{s,s'\in\Z}
    \widehat{C}(\omega/\gamma-s)
    \overline{\widehat{C}(\omega/\gamma-s')}
    \ifcoef{(s-s')\tau}.
\end{split}
\end{equation}
Therefore,
\begin{equation}
\label{eq:B expression}
\begin{split}
    B
    &=
    N
    \sum_{s,s'\in\Z}
    \widehat{C}(\omega/\gamma-s)
    \overline{\widehat{C}(\omega/\gamma-s')}
    \widehat{\nu}_{s'-s}
    +o(N).
\end{split}
\end{equation}

Substituting the above expressions for $A$ and $B$ into
\eqref{eq:optimal h definition with A and B} and
\eqref{eq:error bound depending on t_n} respectively, and noticing that $\overline{\widehat{\nu}_{-s}} = \widehat{\nu}_s$ since $\nu$ is real, we obtain
\begin{equation}
    h_\nu^*(\omega;C)
    =
    \frac{
    \sum_{s\in\Z}
    \widehat{\nu}_s
    \overline{\widehat{C}(\omega/\gamma-s)}
    }{
    \sum_{s,s'\in\Z}
    \widehat{\nu}_{s'-s}
    \widehat{C}(\omega/\gamma-s)
    \overline{\widehat{C}(\omega/\gamma-s')}
    }
    +o(1),
\end{equation}
and
\begin{equation}
\begin{split}
    \ell_\nu(\omega, C) = U^*(\omega,C)
    &=
    N\left(
    1-
    \frac{
    \left|
    \sum_{s\in\Z}
    \widehat{\nu}_{-s}
    \widehat{C}(\omega/\gamma-s)
    \right|^2
    }{
    \sum_{s,s'\in\Z}
    \widehat{\nu}_{s'-s}
    \widehat{C}(\omega/\gamma-s)
    \overline{\widehat{C}(\omega/\gamma-s')}
    }
    \right)
    +o(N).
\end{split}
\end{equation}
This proves the result.
\end{proof}
    Before concluding this section, we give the following definition.
    \begin{definition}[Error profile]
         Under the hypothesis of \Cref{thm:error bound}, and for a fixed a gridding kernel $C \in \cont^2(\R)$ with support in $[-W, W]$ and satisfying \Cref{rmk:C is nu not zero}, we define the function
         \begin{equation}
             \ell_\nu(\cdot, C): \fov \longrightarrow \R \quad\text{s.t.}\quad \omega \mapsto \ell_\nu(\omega, C)
         \end{equation}
         the \emph{error profile} of $C$.
    \end{definition}

\subsection{Applicability of the equidistribution hypothesis}
\label{ch:Applicability of the equidistribution hypothesis}
\Cref{hyp:time samples are equidistributed - formal} is somewhat technical and, in principle, it is not obvious that it can be actually found in real-world applications. In this section we try to give some evidence of the meaningfulness of the assumption.

We start from a classical source of non-uniformity in the time sampling, that is, the jitter caused by an analog-to-digital converter (ADC) on signals \cite{wagdy_effect_1987, gans_measurement_1983, wagdy_effect_1990}. In this case, sampling is usually modelled by a random variable
\begin{equation}
    t_n = n + \epsilon_n, \qquad \forall n = 0, \dots, N - 1,
\end{equation}
where $\epsilon_n \sim \mathcal{X}$ and $\mathcal{X}$ is some known (or approximately known) random variable \cite{wagdy_effect_1987,wagdy_effect_1990}, like $\mathcal{X} \sim \unif(-\varepsilon, \varepsilon)$ (the uniform probability distribution on the interval $[-\varepsilon, \varepsilon]$) or $\mathcal{X} \sim \normal(0, \varepsilon^2)$ (the Gaussian distribution with variance $\varepsilon^2$). For a fixed $\gamma \geq 1$, we then have
\begin{equation}
    \fracpart{\gamma t_n} = \fracpart{\gamma n + \gamma \epsilon_n}.
\end{equation}
Suppose for simplicity that $\gamma \in \Z$. Then $\fracpart{\gamma n} = 0$ and thus we get
\begin{equation}
    \fracpart{\gamma t_n} = \fracpart{\gamma \epsilon_n} \sim \fracpart{\gamma \mathcal{X}},
\end{equation}
whose probability distribution can be easily deduced from that of $\epsilon_n$. If $\nu$ is the probability measure induced on $\fov$ by $\fracpart{\gamma \mathcal{X}}$, by a simple Monte Carlo argument we see that 
\begin{equation}
    \frac{1}{N}\sum_{n < N} \delta_{\fracpart{\gamma \epsilon_n}} \longrightarrow \nu
\end{equation}
in the weak-$*$ topology, and so \Cref{hyp:time samples are equidistributed - formal} holds. Furthermore, the probability distribution of $\fracpart{\gamma \mathcal{X}}$ can be easily found using the characteristic function of $\gamma \mathcal{X}$. Indeed, for any $k \in \Z$ we have
\begin{equation}
    \mathbb{E}[\fcoef{k  \gamma\mathcal{X}}] = \mathbb{E}[\fcoef{(k\fracpart{\gamma \mathcal{X}} + k\intpart{\gamma \mathcal{X}})}] = \mathbb{E}[\fcoef{k  \fracpart{\gamma\mathcal{X}}}] = \widehat{\nu}_k.   
\end{equation}
Thus, the characteristic function of $\fracpart{\gamma \mathcal{X}}$ is the restriction of the characteristic function of $\gamma \mathcal{X}$ to the integers.

Suppose now that $\gamma \in \Q \setminus \Z$, say $\gamma = \frac{p}{q}$ for $p, q$ coprime and positive integers. Let $\phi$ be the characteristic function of $\mathcal{X}$. We have
\begin{equation}
   \mathbb{E}[\fcoef{k  \gamma t_n}] = \mathbb{E}\big[\exp\big(-\frac{2 \pi ipk(n + \epsilon_n)}{q}\big)\big] = \ffcoef{pkn}{q} \mathbb{E}\big[\exp\big(-\frac{2 \pi i kp \epsilon_n}{q}\big)\big] = \ffcoef{pkn}{q} \phi\big(\frac{kp}{q}\big).
\end{equation}
The last expression is the characteristic function of the random variable $\mathcal{X}_{r} = \frac{r}{q} + \frac{p}{q}\mathcal{X}$, where $r$ is the remainder of $pn$ modulo $q$. Furthermore, we trivially have $\mathbb{E}[\fcoef{k  \gamma t_n}] = \mathbb{E}[\fcoef{k  \fracpart{\gamma t_n}}]$. Therefore, for each $r = 0, \dots, q - 1$, the variables $\fracpart{\gamma t_{r + jq}}$ are distributed as $\fracpart{\mathcal{X}_r}$, for all $j \in \Z$. Let $\nu_r$ be the probability measure on $\fov$ induced by $\fracpart{\mathcal{X}_r}$. Supposing for simplicity that $N \equiv 0 \mod q$, by Monte Carlo we finally obtain
\begin{equation}
    \frac{1}{N}\sum_{n < N} \delta_{\fracpart{\gamma t_n}} = \frac{1}{N} \sum_{r = 0}^{q - 1} \sum_{j < N/q}\delta_{\fracpart{\gamma t_{r + jq}}} \longrightarrow \frac{N/q}{N} \sum_{r = 0}^{q - 1} \nu_r = \frac{1}{q}\sum_{r = 0}^{q - 1} \nu_r,
\end{equation}
in the weak-$*$ topology. Therefore, also in this case \Cref{hyp:time samples are equidistributed - formal} holds.

\null

\noindent Another example of non-uniformity is offered by the theory of Swept Source Optical Coherence Tomography (SS-OCT) \cite{rivet_complex_2016}. In this case, the "time samples" (actually, the samples in the $k$-space) are distributed accordingly to a certain invertible (and usually differentiable) function $g:[0, T] \rightarrow \R$, that is, $t_n = g(nT/N)$, for $n = 0, \dots, N-1$. Given a function $\psi \in \cont([0, T))$, we have
\begin{equation}
    \frac{1}{N}\sum_{n < N} \psi(t)\delta_{\gamma g(nT/N)}(t) = \frac{1}{N} \sum_{n < N} \psi(\gamma g(nT/N)),
\end{equation}
which is a Riemann sum whose limit for $N \rightarrow \infty$ is
\begin{equation}
     \frac{1}{T}\int_0^T \psi(\gamma g(t)) dt.
\end{equation}
Thus, the average of the Dirac deltas converges in the weak-$*$ topology to the pushforward $\nu$ of the Lebesgue measure on $[0, T)$ by the map $\gamma g$. This is a measure on the interval $[g(0), g(T))$, whose reduction modulo $1$ gives a measure $\Tilde{\nu}$ on $\fov$. By \cite[Theorem 2.7]{billingsley_convergence_1999} we get
\begin{equation}
    \frac{1}{N} \sum_{n < N} \delta_{\fracpart{\gamma g(nT/N)}} \longrightarrow \Tilde{\nu}
\end{equation}
in the weak-$*$ topology as $N \rightarrow \infty$. Therefore, \Cref{hyp:time samples are equidistributed - formal} holds. With arguments similar to those used in the first example, one can prove that $\Tilde{\nu}$ can be effectively computed starting from $g$. Naturally, the same approach can be used by any application in which the sampling is governed by a function $g$.

\null

\noindent The above examples serve as a proof of the applicability of \Cref{hyp:time samples are equidistributed - formal}. More real-world applications shall be studied by the authors in future work.

\section{Vector optimization for the adjoint NUFFT}
\label{ch:Applying vector optimization to ell(x)}
Once the limit measure $\nu$ has been fixed, \Cref{thm:error bound} gives us an error bound $\ell_\nu(\omega, C)$ for the adjoint NUFFT that depends only on the frequency $\omega \in \fov$ and the chosen gridding kernel $C \in \cont^2(\R)$ with support contained in $[-W, W]$. Therefore, minimizing the value of $\ell_\nu(\omega, C)$ for a certain $\omega_0 \in \fov$ is equivalent to finding the gridding kernel $C$ whose approximation error at $\omega_0$ is minimal. This consideration suggests to use the theory of vector optimization (VO) \cite{jahn_vector_2011} to control the error profile $\ell_\nu$ \textit{locally}. 

This section is devoted to this task. In \Cref{ch:Background on vector optimization} we briefly recall some of the main concepts of VO, with a particular emphasis on $L^p$ spaces. In \Cref{ch:Definition of the optimization problem} we formally define the optimization problem we want to solve for $\ell_\nu^*(\omega; C)$, exploiting the theory shown in \Cref{ch:Background on vector optimization}. In \Cref{ch:Lambda is continuous} we prove a continuity property related to $\ell^*(\omega; C)$ which is useful for proving the existence of (approximate) solutions. In \Cref{ch:Existence of solutions for the minimization problem} we discuss the existence of Pareto optimal solutions. 

\subsection{Background on vector optimization}
\label{ch:Background on vector optimization}
A classical optimization problem is usually stated as a minimization problem: given a function $f:X \rightarrow \R$, with $X$ some set of possible solutions, find an element $x \in X$ that realizes the minimum of $f$, that is, $f(x) \leq f(x')$, for all $x' \in X$. The theory of vector optimization was developed to formalize a notion of optimization problem in which the codomain $\R$ is substituted by a more general real vector space endowed with a (partial) ordering. VO problems typically do not possess a unique solution, and different solutions typically do not give the same minimal value for $f$. We shall see how this theory fits our needs to optimize the error profile of the adjoint NUFFT in \Cref{ch:Applying vector optimization to ell(x)}. We now give a brief introduction to VO.

\null

\noindent In classical optimization theory, the definition of minimum and maximum are based on the total ordering defined on $\R$: $x < y$ if and only if $y - x \in (0, \infty)$. To generalize these notions to a vector space $Y$, we thus need a subset of $Y$ that represents an analogous of $(0, \infty)$. 

Let $Y$ be a real topological linear space. A subset $D\subseteq Y$ is called a \textit{cone} if it is closed under multiplication by nonnegative scalars. It is a \textit{proper} cone if $\{0\} \not = D$, and \emph{pointed} if $D \cap -D = \{0\}$. Finally, $D$ is \textit{convex} if $D + D \subseteq D$\footnote{The reader may be more accustomed to a notion of convexity based on convex combinations of elements. However, it can be shown that the definition given here is equivalent, besides being the one commonly found in VO literature.}. In this paper, we will always assume that $D$ is a proper pointed convex cone. Any proper pointed convex cone $D$ induces a partial ordering on $Y$ defined by
\begin{equation*}
    y_1 \preceq y_2 \Longleftrightarrow y_2 - y_1 \in D, \quad \forall y_1, y_2 \in Y.
\end{equation*}
The inequality is strict if $y_1 \neq y_2$, and we write $y_1 \prec y_2$.
\begin{example}
    Let $Y = \R^N$, with $N \geq 1$. Then
    \begin{equation*}
        D = \{ (x_1, \dots, x_N) \in \R^N \st x_1 \geq 0, \dots, x_N \geq 0 \}
    \end{equation*}
    is a proper pointed convex cone. The ordering induced by $D$ is given by $\bm{x} \preceq \bm{y}$ if and only if $x_n \leq y_n$, for all $n = 1, \dots, N$.
\end{example}
\begin{example}[Example 1.51, \cite{jahn_vector_2011}]
\label{ex:Cone for Lp(Omega)}
    Let $Y = L_\R^p(\Omega)$, with $\Omega$ a measure space and $p \in (0, \infty]$. Then
    \begin{equation*}
        D = \{f \in L_\R^p(\Omega) \st f(\omega) \geq 0 \text{ for almost all } \omega \in \Omega \}
    \end{equation*}
    is a proper pointed convex cone. The ordering induced by $D$ is given by $f \preceq g$ if and only if $f(x) \leq g(x)$ for almost all $x \in \Omega$.
\end{example}

Let $Z \subseteq Y$, $Z \not = \varnothing$. Consider the following problem:
\begin{equation}
    \label{prbl:abstract VO problem}
    \begin{split}
        &\text{minimize } z \in Z\subseteq Y \\
        &\text{with respect to the ordering } \preceq \text{ defined by $D\subseteq Y$.}
    \end{split}
\end{equation}
Problems of this type are called \textit{vector optimization problems}. Frequently, $Z$ is the image of a function $f:S \rightarrow Y$, where $S\subseteq X$ and $X$ is a real topological linear space. As $\preceq$ is not a total order in general, we cannot expect to find a unique solution, but rather a set of minimal points. 
\begin{definition}[Pareto efficient points]
    A point $z^* \in Z$ is \emph{Pareto efficient} (or \emph{Pareto optimal}) if it is minimal, that is, there does not exist another $z \in Z$, $z \not = z^*$, such that $z \prec z^*$. The set of all Pareto optimal points is called the \emph{Pareto front} of $Z$.
\end{definition}

The study of Pareto optimal solutions is usually pursued by \textit{scalarization}, that is, by converting the original problem into one or more optimization problems with only one cost function. This concept has proven to be useful both from a theoretical and a practical point of view. Scalarization is achieved by minimizing monotonic (possibly nonlinear) functionals of $Y$.
\begin{definition}[$D$-increasing functionals]
\label{def:D-increasing functional}
    A functional $\phi:Y \rightarrow \R$ is 
    \begin{enumerate}[label=(\roman*)]
        \item \emph{$D$-increasing} if $y_1 \preceq y_2 \Rightarrow \phi(y_1) \leq \phi(y_2)$,
        \item \emph{strongly $D$-increasing} if $y_1 \prec y_2 \Rightarrow \phi(y_1) < \phi(y_2)$,
    \end{enumerate}
    for all $y_1, y_2 \in Y$.
\end{definition}
It is easy to see that if $\phi$ is strongly $D$-increasing, then it is also $D$-increasing.
\begin{example}
    Let $Y = L^p(\Omega)$ and $D$ as in \Cref{ex:Cone for Lp(Omega)}. If $p < \infty$, then the functional $\|\cdot\|_p^p$ is a strongly $D$-increasing functional. 
\end{example}

\begin{proposition}[Lemma 5.14, \cite{jahn_vector_2011}]
    \label{prop:Pareto sufficient conditions}
    Let $\phi:Y \rightarrow \R$ be a functional and suppose there exists $z^* \in Z$ such that 
    \begin{equation}
    \label{eq:abstract minimality for functional}
       \phi(z^*) \leq \phi(z),  \quad \forall z \in Z. 
    \end{equation}
    Then: 
    \begin{enumerate}[label=(\roman*)]
        \item if $\phi$ is $D$-increasing and $z^*$ is uniquely determined by \cref{eq:abstract minimality for functional}, then $z^*$ is a Pareto optimal element of $Z$;
        \item if $\phi$ is strongly $D$-increasing, then $z^*$ is a Pareto optimal element of $Z$.
    \end{enumerate}
\end{proposition}

\null

\noindent We refer, {\em e.g.},  to \cite{jahn_vector_2011,luc_theory_1989, jahn_unified_2023} for an extended background on these topics. 

\subsubsection*{The $L^p(\Omega)$ case}
\label{sec: Lp case}
As we shall see in \Cref{ch:Applying vector optimization to ell(x)}, in this paper we are particularly interested in applying VO to the case of $L^p$ spaces, to which the error profile $\ell_\nu(\omega, C)$ belongs (for fixed $C$). Suppose we are given a measure space $(\Omega, \Sigma, \mu)$, where $\Omega$ is a topological space, $\Sigma$ is the Borel $\sigma$-algebra on $\Omega$ and $\mu$ is a complete measure $\Sigma \rightarrow [0, +\infty]$. Let $H$ be a (real) topological vector space $H$ and $S\subseteq H, S \neq \varnothing$. Suppose also that we are given a cost function
\begin{equation}
   L : \Omega \times S \longrightarrow \R 
\end{equation}
such that, for a fixed $p \in (0, \infty)$ and for all $h \in S$, the function $L_h \equiv L(\cdot, h):\Omega \rightarrow \R$ belongs to $L^p(\Omega)$. If we define the operator
\begin{equation}
    \Lambda:S \longrightarrow L^p(\Omega) \quad \text{such that}\quad \Lambda(h) = L_h = L(\cdot, h),
\end{equation}
the connection with vector optimization theory should become evident. In fact, we have seen in \Cref{ex:Cone for Lp(Omega)} that $L^p(\Omega)$ comes with a natural choice of a proper pointed convex cone $D$. Therefore, setting $Z = \Lambda(S)$ (the \textit{feasible region}), we are in the same setting described in problem (\ref{prbl:abstract VO problem}). 

In virtue of \Cref{prop:Pareto sufficient conditions}, we are particularly interested in finding functionals over $L^p(\Omega)$ that are $D$-increasing. In this context, Wierzbicki \cite{wierzbick_basic_1977} developed an approach based on a particular kind of functionals. 
\begin{definition}[Penalty scalarizing functionals]
\label{def:penalty scalarizing functionals}
 Fix $\eta \in Y$ and $\rho > 1$. The \emph{$L^p$-penalty scalarizing functional} defined by $\eta$ and $\rho$ is the functional
 \begin{equation}
    \label{eq:SAF definition}
    s_{\rho, \eta}: L^p(\Omega) \longrightarrow \R \quad\text{s.t.}\quad s_{\rho, \eta}(f) = -\|f - \eta \|^p_p + \rho \|(f - \eta)_+\|^p_p,\quad\forall f \in L^p(\Omega),
\end{equation}
where $(f - \eta)_+ = \max\{0, f - \eta \}$. The point $\eta$ is called the \emph{target point}, and $\rho$ is called the \emph{penalty coefficient}. 
\end{definition}
The idea behind the definition of $s_{\rho, \eta}$ is to look for elements in the Pareto front that minimize (or maximize) the distance from $\eta$. More precisely, suppose $\eta \preceq z$ for any $z \in Z$. Then $(f - \eta)_+ = f - \eta$, and since $\rho > 1$, $s_{\rho, \eta}$ boils down to the classical $p$-norm: that is, we look for the elements on the Pareto front with \textit{minimal} $p$-distance from $\eta$. On the other hand, if there exists $z \in Z$ such that $z \preceq \eta$, then we should look for the elements in $z \in Z$ such that $z \preceq \eta$ and their $p$-distance from $\eta$ is maximal (so that we are certain that they lie on the Pareto front). In any case, we should arrive at a Pareto optimal solution which is somewhat derived from $\eta$.

The following results formalize the above idea.
\begin{proposition}
\label{prop:Penalty function is strongly D-increasing}
    Let $s_{\rho, \eta}:L^p(\Omega) \rightarrow \R$ as above, with $p \in (0, \infty)$. Then $s_{\rho, \eta}$ is a strongly $D$-increasing functional.
\end{proposition}
\begin{proof}
    Let $\varphi(f) = \|f\|_p^p$, so $\varphi(f) = \varphi(f_+) + \varphi(f_-)$ and hence
    \[
        s_{\rho,\eta}(f) = (\rho-1)\,\varphi\big((f-\eta)_+\big) - \varphi\big((f-\eta)_-\big).
    \]
    Write $h = f-\eta$, $k = g-\eta$. If $f \preceq g$, then $h \le k$ a.e., so $h_+ \le k_+$ and $k_- \le h_-$ a.e.; since $t \mapsto t^p$ is non-decreasing on $[0,\infty)$, $\varphi(h_+) \le \varphi(k_+)$ and $\varphi(k_-) \le \varphi(h_-)$, giving
    \[
        s_{\rho,\eta}(g) - s_{\rho,\eta}(f) = (\rho-1)\big[\varphi(k_+)-\varphi(h_+)\big] + \big[\varphi(h_-)-\varphi(k_-)\big] \ge 0.
    \]
    If moreover $f \prec g$, then $h < k$ on a set of positive measure. For reals $a<b$, $a_+=b_+$ and $a_-=b_-$ cannot both hold (else $a = a_+-a_- = b_+-b_- = b$), so at each such point at least one of $h_+ < k_+$, $k_- < h_-$ holds; hence one of these is strict on a set of positive measure. Since $t \mapsto t^p$ is strictly increasing on $[0,\infty)$ for $p>0$, this makes one bracketed term above strictly positive while the other stays $\ge 0$, so $s_{\rho,\eta}(g) > s_{\rho,\eta}(f)$.

    Thus $s_{\rho,\eta}$ is strongly $D$-increasing.
\end{proof}

\subsection{The optimization problem for \texorpdfstring{$\ell_\nu(\omega, C)$}{l(omega,C)}}
\label{ch:Definition of the optimization problem}
The theory of vector optimization can be used to find gridding kernels tailored to specific applications. Indeed, given a gridding kernel $C$, we shall see that its error profile function $\ell_\nu(\cdot, C)$ belongs to the space $L^p_\R(\fov)$, for all $p \in (0, \infty]$ (from now on, we drop the little-o error in the definition of $\ell_\nu(\omega, C)$). We saw in \Cref{ex:Cone for Lp(Omega)} that the space $L^p_\R(\fov)$ is naturally endowed with a partial ordering defined elementwise. In particular, a minimal element will be a solution such that \textit{if we try to decrease the error at some subset of frequencies by modifying the solution, this will necessarily cause an increase of the error at the remaining set of frequencies}. So, a Pareto optimal solution represents a minimal error in a very intuitive sense.

Formally, let $\Omega = \fov$ endowed with the Lebesgue measure, and
\begin{equation*}
    H = \{C \in \cont^2(\R) \st \supp(C) \subseteq I_W \}.
\end{equation*}
Define
\begin{equation*}
    S = \{C \in H \st \|C\|_2 = 1 \text{ and $C$ satisfies \Cref{rmk:C is nu not zero}} \} \subseteq H.
\end{equation*}
The choice of $S$ is justified by the fact that $\ell_\nu(\omega, C) = \ell_\nu(\omega, \alpha C)$, for any $\alpha \in \C \setminus\{0\}$. The operator defined below corresponds to the cost function we introduced in \Cref{ch:Background on vector optimization} when dealing with the case $H = L^p(\Omega)$ for some $p \in (0, \infty)$. 
\begin{definition}
Fix any $p \in (0, \infty]$ and a probability measure $\nu \in \fov$. We define the \emph{$p$-th error profile operator} with respect to $\nu$ as the (nonlinear) operator
\begin{equation}
    \label{eq:loss gridding definition}
    \begin{split}
    \Lambda_{p, \nu}: S \longrightarrow L_\R^p(\Omega) \quad \text{ s.t. } \quad \Lambda_{p, \nu} C=\Lambda_{p, \nu} C(\omega) = \ell_\nu(\omega, C). 
    \end{split}
\end{equation}
\end{definition}
We need have the following preliminary result.
\begin{lemma}
    \label{lemma:lambda is well defined}
    The $p$-th error profile operator $\Lambda_{p, \nu}$ is well defined for any $p \in (0, \infty]$.
\end{lemma}
\begin{proof}
First, notice that $\Lambda_{p, \nu} C(\omega) = \ell_\nu(\omega, C)$ is a measurable function. Furthermore, if $C \neq 0$, the set of points $\omega \in \fov$ such that the denominator is null has zero measure. Indeed, we saw in the proof of \Cref{thm:error bound} the denominator is almost everywhere positive, provided that $C$ satisfies \Cref{rmk:C is nu not zero}.

For any $\omega$ such that the denominator is not null, we have trivially $\Lambda_{p, \nu} C (\omega) \leq N$, since both the numerator and the denominator in the fraction are positive. Furthermore, by Cauchy-Schwartz inequality we get
\begin{equation*}
\begin{split}
    \left|\sum_{s\in\Z}
    \widehat{\nu}_{-s}
    \widehat{C}(\omega/\gamma-s)
    \right|^2 & = \left|\int_{-1/2}^{1/2} \sum_{s\in\Z}
    \widehat{C}(\omega/\gamma-s) \ifcoef{s \tau} d\nu(\tau)\right|^2 \\
    & \leq
    \int_{-1/2}^{1/2} \left|\sum_{s\in\Z}
    \widehat{C}(\omega/\gamma-s) \ifcoef{s \tau}\right|^2 d\nu(\tau) = \sum_{s,s'\in\Z}
    \widehat{\nu}_{s'-s}
    \widehat{C}(\omega/\gamma-s)
    \overline{\widehat{C}(\omega/\gamma-s')}.
\end{split}
\end{equation*}
Therefore, $\Lambda_{p, \nu} C (\omega) \geq 0$. Thus, $\|\Lambda_{p, \nu} C\|_\infty \leq N$, and, \textit{a fortiori}, $\|\Lambda_{p, \nu} C \|_p \leq N$.
\end{proof}

Since $L^p(\Omega)$ is a partially ordered topological real space (see \Cref{ex:Cone for Lp(Omega)}), we can apply the vector optimization machinery. We thus define the following vector optimization problem:
\begin{equation}
    \label{prbl:lambda minimization problem}
\begin{split}
    & \text{minimize } \Lambda_{p, \nu} C \in \Lambda_{p, \nu}(S)\subseteq H \\
    & \text{with respect to the partial ordering $\preceq$ induced by \Cref{ex:Cone for Lp(Omega)}}\\
    & \text{subject to } C \in S.
\end{split}    
\end{equation}

\subsection{Continuity of \texorpdfstring{$\Lambda_p$}{Lambda}}
\label{ch:Lambda is continuous}
From now on we shall assume that $\cont^2(I_W)$ is endowed with the classical norm
\begin{equation*}
    \|C\|_{\cont^2(I_W)} = \|C\|_\infty + \|C'\|_\infty + \|C''\|_\infty.
\end{equation*}
The following results are needed for the proof of the continuity of $\Lambda_{p, \nu}$.
\begin{lemma}
\label{lemma:Q is continuous}
    Let $\nu$ be a probability measure on $\fov$ absolutely continuous with respect to the Lebesgue measure. The quadratic operator
    \begin{equation}
        Q_\nu: H \longrightarrow L^\infty(\T) \qquad\text{s.t.}\qquad C \mapsto Q_\nu C = Q_\nu C(\omega) = \sum_{s, s' \in \Z} \widehat{\nu}_{s' - s} \widehat{C}(\omega/\gamma - s) \overline{\widehat{C}(\omega/\gamma - s')}
    \end{equation}
    is well-defined and continuous.
\end{lemma}

\begin{proof}
    First, we prove the well-definedness. From the proof of \Cref{thm:error bound}, recall that, for each $\omega \in \fov$,
    \begin{equation*}
    \begin{split}
        \sum_{s, s' \in \Z} \widehat{\nu}_{s' - s} \widehat{C}(\omega/\gamma - s)\overline{\widehat{C}(\omega/\gamma - s')} 
        & = \int_{-1/2}^{1/2} \sum_{s, s' \in \Z}  \widehat{C}(\omega/\gamma - s)\overline{\widehat{C}(\omega/\gamma - s')} \fcoef{(s' - s)\tau}  d\nu(\tau) \\
        & = \int_{-1/2}^{1/2}|D_{C,\omega}(\tau)|^2 d\nu(\tau),  
    \end{split}
    \end{equation*}
    where
    \begin{equation*}
        D_{C,\omega}(\tau) := \sum_{s \in \Z} \widehat{C}(\omega/\gamma - s) \ifcoef{s\tau} = \sum_{k \in \Z} C(k - \tau) \fcoef{\omega(k - \tau)/\gamma}, \qquad \forall \tau \in \fov.
    \end{equation*}
    In particular,
    \begin{equation*}
        |D_{C,\omega}(\tau)| \leq \sum_{k \in \Z} |C(k - \tau)|.
    \end{equation*}
    Since $C$ has support in $I_W$ and $\tau \in \fov$ ranges over a finite interval, there exist only finitely many indices $k$ such that, for all $\tau$, $C(k - \tau) \neq 0$. Thus there exist $c > 0$, depending only on $I_W$, such that
    \begin{equation*}
        \sum_{k \in \Z} |C(k - \tau)| \leq c \|C\|_\infty \leq c \|C\|_{\cont^2(I_W)}, \qquad \forall \tau \in \fov,
    \end{equation*}
    and consequently
    \begin{equation*}
        |D_{C,\omega}(\tau)| \le c\|C\|_{\cont^2(I_W)}, \qquad \forall \tau \in \fov.
    \end{equation*}
    Since $\nu$ is a probability measure,
    \begin{equation*}
        \int_{-1/2}^{1/2}|D_{C,\omega}(\tau)|^2 d\nu(\tau) \leq c^2 \|C\|_{\cont^2(I_W)}^2< \infty, 
    \end{equation*}
    for all $\omega\in \fov$. Therefore $Q_\nu C \in L^\infty(\T)$.

    Now we prove the continuity. For $x, y \in \C$ we have the elementary identity $|x|^2 - |y|^2 = (x-y)\overline{x} + y\overline{(x-y)}$, and thus
    \begin{equation}
    \label{eq:polar inequality}
        \big||x|^2 - |y|^2\big| \le |x - y|\big(|x| + |y|\big).
    \end{equation}
    Applying this inequality with $x = D_{C,\omega}(\tau)$, $y = D_{\widetilde C, \omega}(\tau)$, for arbitrary $C, \widetilde{C} \in \cont^2(I_W)$, we get
    \begin{equation}
    \label{eq:Q difference}
    \begin{split}
        |Q_\nu C(\omega) - Q_\nu \widetilde{C}(\omega)| & \leq \int_{-1/2}^{1/2} \left||D_{C, \omega}(\tau)|^2 - |D_{\widetilde{C}, \omega}(\tau)|^2\right|d\nu(\tau) \\
        & \le \int_{-1/2}^{1/2} |D_{C,\omega}(\tau) - D_{\widetilde C, \omega}(\tau)|\Big(|D_{C,\omega}(\tau)| + |D_{\widetilde C, \omega}(\tau)|\Big)\, d\nu(\tau).
    \end{split}
    \end{equation}
    As in the first part,
    \begin{equation*}
    \begin{split}
        |D_{C, \omega}(\tau) - D_{\widetilde{C}, \omega}(\tau)| & = \left|\sum_{k \in \Z}\left(C(k - \tau) - \widetilde{C}(k - \tau)\right)\ifcoef{\omega(k - \tau)/\gamma}\right|\\
        &\leq \sum_{k \in \Z}|C(k-\tau) - \widetilde C(k-\tau)| \le c\|C - \widetilde{C}\|_\infty \leq c\|C - \widetilde{C}\|_{\cont^2(I_W)},
    \end{split}
    \end{equation*}
    and, exactly as before,
    \begin{equation*}
        |D_{C,\omega}(\tau)| \le c\|C\|_{\cont^2(I_W)}, \qquad |D_{\widetilde C,\omega}(\tau)| \le c\|\widetilde C\|_{\cont^2(I_W)}, \qquad \forall \tau \in \fov.
    \end{equation*}
    Substituting these bounds into \eqref{eq:Q difference} and using that $\nu(\T) = 1$, we obtain
    \begin{equation*}
        |Q_\nu C(\omega) - Q_\nu \widetilde{C}(\omega)| \le c^2 \Big(\|C\|_{\cont^2(I_W)} + \|\widetilde C\|_{\cont^2(I_W)}\Big)\|C - \widetilde{C}\|_{\cont^2(I_W)},
    \end{equation*}
    that is,
    \begin{equation}
    \label{eq:Q bound}
        \|Q_\nu C - Q_\nu \widetilde{C}\|_\infty \le c^2 \Big(\|C\|_{\cont^2(I_W)} + \|\widetilde C\|_{\cont^2(I_W)}\Big)\|C - \widetilde{C}\|_{\cont^2(I_W)}.
    \end{equation}
    Fix $C \in \cont^2(I_W)$ and suppose $\|C - \widetilde{C}\|_{\cont^2(I_W)} < \delta$. Then, trivially,
    \begin{equation*}
        \|\widetilde{C}\|_{\cont^2(I_W)} \leq \|C\|_{\cont^2(I_W)} + \delta.
    \end{equation*}
    Substituting in \eqref{eq:Q bound} we obtain
    \begin{equation*}
        \|Q_\nu C - Q_\nu \widetilde{C}\|_\infty \le c^2 \Big(2\|C\|_{\cont^2(I_W)} + \delta\Big)\delta.
    \end{equation*}
    Since $C$ is fixed and $\delta$ is arbitrary, $Q_\nu$ is continuous at $C$ and we get the result.
\end{proof}

\begin{lemma}
\label{lemma:Q dagger is continuous}
Let $\nu$ be a probability measure on $\fov$ absolutely continuous with respect to the Lebesgue measure. The quadratic operator
    \begin{equation}
        Q^\dagger_\nu: H \longrightarrow L^\infty(\T) \qquad\text{s.t.}\qquad C \mapsto Q^\dagger_\nu C = Q^\dagger_\nu C(\omega) = \left|\sum_{s \in \Z} \widehat{\nu}_{-s} \widehat{C}(\omega/\gamma - s)\right|^2
    \end{equation}
    is well-defined and continuous.
\end{lemma}
\begin{proof}
     The result follows as in the previous lemma. To prove the continuity, we apply \eqref{eq:polar inequality} to the integrals, rather than their arguments.
\end{proof}

\begin{lemma}
    \label{lemma:ratio pointwise convergence}
    Let $\nu$ be a probability measure on $\T$. Let $C_n \in S$ be a sequence of gridding kernels converging to $C \in S$ in the $\|\cdot\|_{\cont^2(I_W)}$ norm. The sequence of ratios $Q^\dagger_\nu C_n / Q_\nu C_n$ converges pointwise to $Q^\dagger_\nu C / Q_\nu C$ almost everywhere.
\end{lemma}
\begin{proof}
    Let $X\subseteq \fov$ be the intersection of all the domains of the functions $Q^\dagger_\nu C_n/Q_\nu C_n$ and $Q^\dagger_\nu C/Q_\nu C$ and let $\omega\in X$. We know that $\fov\setminus X$ has measure zero by \Cref{lemma:lambda is well defined}, so it is sufficient to prove the result on $X$. We have
    \begin{equation}
    \label{eq:ratio first estimate}
        \begin{split}
            \left|\frac{Q^\dagger_\nu C(\omega)}{Q_\nu C(\omega)} - \frac{Q^\dagger_\nu C_n(\omega)}{Q_\nu C_n(\omega)}\right|
            & = \left|\frac{Q_\nu C_n(\omega)\cdot(Q^\dagger_\nu C(\omega) - Q^\dagger_\nu C_n(\omega)) +  Q^\dagger_\nu C_n(\omega) \cdot (Q_\nu C_n(\omega)  - Q_\nu C(\omega))}{Q_\nu C_n(\omega) \cdot Q_\nu C(\omega)}\right| \\
            & \leq \left|\frac{Q^\dagger_\nu C(\omega) - Q^\dagger_\nu C_n(\omega)}{Q_\nu C(\omega)}\right| + \left| \frac{Q^\dagger_\nu C_n(\omega)}{Q_\nu C_n(\omega)} \cdot \frac{Q_\nu C_n(\omega)  - Q_\nu C(\omega)}{Q_\nu C(\omega)}\right| \\
            & \leq \left|\frac{Q^\dagger_\nu C(\omega) - Q^\dagger_\nu C_n(\omega)}{Q_\nu C(\omega)}\right| + \left| \frac{Q_\nu C_n(\omega)  - Q_\nu C(\omega)}{Q_\nu C(\omega)}\right|.
        \end{split}
    \end{equation}
    Notice that $\left|\frac{Q^\dagger_\nu C_n(\omega)}{Q_\nu C_n(\omega)}\right| \leq 1$ by \Cref{lemma:lambda is well defined}. Now fix $\varepsilon > 0$ and let $N_0\in\N$ be such that for any $n \geq N_0$ we have 
    \begin{equation*}
        |Q^\dagger_\nu C_n(\omega) - Q^\dagger_\nu C(\omega)| < \frac{\varepsilon}{2}|Q_\nu C(\omega)| \quad \text{and} \quad |Q_\nu C_n(\omega) - Q_\nu C(\omega)| < \frac{\varepsilon}{2}|Q_\nu C(\omega)|.
    \end{equation*}
    Such $N_0$ exists by \Cref{lemma:Q is continuous} and \Cref{lemma:Q dagger is continuous}. Using these two inequalities in the last expression of \Cref{eq:ratio first estimate}, we get
    \begin{equation*}
        \left|\frac{Q^\dagger_\nu C(\omega)}{Q_\nu C(\omega)} - \frac{Q^\dagger_\nu C_n(\omega)}{Q_\nu C_n(\omega)}\right| \leq \varepsilon, 
    \end{equation*}
    which is the thesis.
\end{proof}

We are now in the position to prove the main result.
\begin{theorem}
    \label{thm:L is continuous}
    The operator $\Lambda_{p,\nu}: S \rightarrow L^p(\fov)$ is continuous for any fixed $p \in (0, \infty)$.
\end{theorem}
\begin{proof}
    Fix $p \in (0, \infty)$. By \Cref{lemma:ratio pointwise convergence}, the operator $\Lambda_{p, \nu}$ is pointwise continuous. Since $|\Lambda_{p, \nu} C_n| \prec N\cdot\bm{1} \in L^p(\fov)$, the result follows from an application of Lebesgue's Dominated Convergence Theorem in $L^p(\Omega)$.
\end{proof}

\subsection{Existence of approximate solutions for the minimization problem}
\label{ch:Existence of solutions for the minimization problem}
Determining whether the problem in \eqref{prbl:lambda minimization problem} has minimal solutions is in general very difficult. Even if we managed to prove the continuity of $\Lambda_{p, \nu}$, the non-compactness (actually, the non-closedness) of $S\subseteq H$ prevents us from deducing the existence of optimal solutions by a standard application of Weierstrass's theorem. To circumvent this issue, we shall limit ourselves to the search for \textit{approximate solutions}, as explained in the following.

\subsubsection{Existence of approximate Pareto optimal solutions}
We simplify the optimization problem by restricting the space of admissible gridding kernels to a suitable finite-dimensional subspace. Before introducing this subspace, however, we have to make a preliminary remark. 

\begin{remark}
\label{rmk:C does not need to be differentiable}
    The differentiability condition imposed to the gridding kernel is likely much stronger than what is actually needed. Indeed, the numerical experiments seem to confirm that one does not even really need a continuous kernel. We suspect that the theory and the result presented above are valid at least for $C \in L^1(I_W)\cap L^2(I_W)$, but the proofs seem to be more difficult, since we lose control of the decay of the Fourier coefficients of $C$. Thus, for this first treatment of the problem we decided to pursue the easiest way. The somewhat cumbersome definition of the finite-dimensional space we give below is a collateral result of this choice. 
\end{remark} 

The finite-dimensional subspace is defined in the following way. To preserve the compact-support and differentiablity requirements of $C$ while ensuring non-degeneracy with respect to the limiting measure (see \Cref{rmk:C is nu not zero}), we first define the bump function
\[
b_W(x)=
\begin{cases}
\exp\!\left(-\dfrac{1}{1-(x/W)^2}\right), & |x|<W,\\
0, & |x|\geq W,
\end{cases}
\]
and, for a fixed integer $m\geq1$,
\[
\Xi_{W,m}(x):=b_W(x)^{1/m}.
\]
The functions $\Xi_{W,m}\in\cont^\infty_c(\R)$ are positive on $(-W,W)$ and real-analytic on $(-W,W)$. Let
\[
\psi_0,\dots,\psi_L
\]
be any linearly independent real-analytic functions on $(-W,W)$, and define
\begin{equation}
\label{eq:bump modulation}
\phi_l(x):=\Xi_{W,m}(x)\psi_l(x),
\qquad l=0,\dots,L.
\end{equation}
We then set
\[
H_L:=\operatorname{span}\{\phi_0,\dots,\phi_L\}.
\]
By construction,
\[
H_L\subseteq\cont^2(\R),
\qquad
\operatorname{supp}(C)\subseteq I_W
\quad\text{for every }C\in H_L.
\]
Moreover, every nonzero element $C\in H_L$ can be written as
\[
C(x)=\Xi_{W,m}(x)\Psi(x),
\qquad
\Psi(x)=\sum_{l=0}^L a_l\psi_l(x),
\]
with $\Psi$ a nonzero real-analytic function on $(-W,W)$. Hence the zero set of $\Psi$ in $(-W,W)$ has Lebesgue measure zero. Since $\eta_{W,m}>0$ on $(-W,W)$, the same is true for $C$. In particular, since $W\geq1$,
\[
C(-\tau)\neq0
\]
for Lebesgue-almost every $\tau\in\fov$, and, under the standing assumption $\nu\ll d\tau$, also for $\nu$-almost every $\tau\in\fov$. Therefore every nonzero element of $H_L$ satisfies the non-degeneracy condition in \Cref{rmk:C is nu not zero}.

Let
\[
S_L=\{C\in H_L:\|C\|_2=1\}
\]
be the unit sphere in $H_L$. Since $H_L$ is finite-dimensional, $S_L$ is compact.

We therefore restrict the vector minimization problem to
\[
\begin{split}
    & \text{minimize } \Lambda_{p, \nu} C\in\Lambda_{p, \nu}(S_L)\subseteq L^p(\fov) \\
    & \text{with respect to the partial ordering $\preceq$ inherited from }L^p(\fov) \\
    & \text{subject to } C\in S_L.
\end{split}
\]

\begin{theorem}[Existence of a Pareto optimal kernel]
\label{thm:solution of vo problem}
Let $p\in(0,\infty)$ and let $S_L$ be the unit sphere of the
finite-dimensional space $H_L$. Then the restriction of problem
\eqref{prbl:lambda minimization problem} to $S_L$ admits at least one
Pareto optimal solution.
\end{theorem}

\begin{proof}
Set
\[
K_L:=\Lambda_{p, \nu}(S_L)\subseteq L^p(\fov).
\]
Recall that $L^p(\fov)$ is naturally endowed with a proper pointed convex cone $D$ (see \Cref{ex:Cone for Lp(Omega)}). Furthermore, $D$ is close. Indeed, let $f_n\in D$ and suppose that
$f_n\to f$ in $L^p(\fov)$. Let
\[
\Omega^- = \{\omega \in\fov \st f(\omega) \leq 0\} \subseteq \fov.
\]
Then, since $|f(\omega)|^p \leq |f_n(\omega) - f(\omega)|^p$ for almost all $\omega \in \Omega^-$,
\[
0 \leq \int_{\Omega^-}|f(\omega)|^p d\omega \leq \int_{\Omega^-}|f_n(\omega) - f(\omega)|^p d\omega \rightarrow 0,
\]
and therefore $f(\omega) = 0$ for almost all $\omega \in \Omega^-$, that is, $f \in D$.

By the continuity of $\Lambda_{p, \nu}$ and the finite-dimensionality of $H_L$, the set $K_L$ is compact. Let $f$ be an element in $K_L$. Then the set $\{f\}-D$ is closed, and thus $K_L \cap (\{f\} - D)$ is compact. \cite[Theorem 6.3(c)]{jahn_vector_2011} then implies that $K_L$ has a minimal element $f^*$.
\end{proof}
A particularly clever choice for the basis of real-analytic functions $\psi_0, \dots, \psi_L$ is given by the finite Slepian basis constituted by the first $L + 1$ prolate spheroidal wave functions \cite{slepian_prolate_1978}, due to its well-known compressing properties in the representation of bandlimited functions. The following result was adapted from \cite{landau_prolate_1962}.
\begin{theorem}[Theorems 1 and 3 of \cite{landau_prolate_1962}, adapted]
Given $\varepsilon>0$, let $E(\varepsilon)$ denote the set of functions
$f\in L^2(I_W)$ satisfying $\|f\|_2=1$ and
\[
    \int_{|x|>1/2}|\widehat f(x)|^2\,dx=\varepsilon^2.
\]
For any fixed $L\in\mathbb{Z}$, $L \geq 0$, consider the minimax problem
\begin{equation}
\min_{\{\phi_l\}_{l=0}^L\subseteq L^2(I_W)}
\max_{f\in E(\varepsilon)}
\min_{\{a_l\}_{l=0}^L}
\big\|
    \widehat f-\sum_{l=0}^L a_l\widehat{\phi}_l
\big\|_2^2.
\end{equation}
An optimal choice of the approximation subspace is
\[
    \operatorname{span}\{\psi_0,\ldots,\psi_L\},
\]
where $\psi_0,\ldots,\psi_L$ are the first $L+1$ prolate spheroidal wave functions (PSWFs). Equivalently, the PSWFs $\psi_0,\ldots,\psi_L$ provide an optimal basis for the minimization problem.
\end{theorem}

This theorem implies that the PSWFs are the best bandlimited basis to represent bandlimited functions. A study on how to improve the value of the constant $C$ is carried out in \cite{landau_prolate_1962}, to which we refer the reader.

\subsubsection{Existence of Pareto optimal solutions with prescribed shapes}
\label{ch:Existence of Pareto optimal solutions with prescribed shapes}
In some scenarios, we are interested only in a restricted portion of the spectrum where we expect the signal information to occur: that is, we can assume \Cref{hyp:ROI of spectrum - informal}. In this case, the use of a gridding kernel minimizing the worst-case error at all frequencies in the field of view $\fov$ could be overkilling, and we would rather prefer a kernel whose error profile is minimal only at the prescribed meaningful frequencies, and whose support is very small, so that the adjoint NUFFT can be computed quicker. We can use vector optimization to address this task. In particular, \Cref{def:penalty scalarizing functionals} can be used to formalized this assumption in the following way.
\begin{hypothesis}[Target error profile]
    \label{hyp:ROI hypothesis - formal}
    Fix $p \in (0, \infty)$. There exists a function $\eta(\omega) \in L^p(\omega)$ representing a \emph{target error profile}, that is, the ideal error profile that an ideal gridding kernel should return for a particular application.  
\end{hypothesis}
The above hypothesis deserves some explanation. The informal intuition behind $\eta$ is that, for each $\omega \in \fov$, the smaller $\eta(\omega)$ is, the more accurate we want the adjoint NUFFT to be at $\omega$. In other words, the smaller $\eta(\omega)$ is, the smaller $\ell_\nu(\omega, C)$ should be. Thus, the function $1/\eta$ could be interpreted as a sort of \textit{utility function} (with the convention that $1/0 = +\infty$), in the sense described, for example, in \cite[Chapter~4]{varian_intermediate_2024}: to each $\omega\in \fov$, it assigns a \textit{rank}, so that areas with higher ranks should be considered more important in terms of accuracy. In an ideal case, an optimal gridding kernel $C$ should satisfy $\Lambda C(\omega) \leq \eta(\omega)$ for all $\omega\in \fov$. However, for a general $\eta$, this inequality cannot be satisfied everywhere, because the error shape operator $\Lambda$ is not onto. More precisely, it can be shown that the $L^1$ norm of $\Lambda C$ is bounded from below by a positive constant related to the uncertainty principle as described in \cite{landau_prolate_1961}. We shall give some real-world examples where \Cref{hyp:ROI hypothesis - formal} holds in \Cref{ch:Applicability of the target error profile hypothesis}.

Even if we cannot find $C$ such that $\Lambda C \preceq \eta$, we can still find a gridding kernel $C \in S$ such that $\Lambda C$ is the Pareto front of $\Lambda(S)$ and, at the same time, has minimal $p$-distance from $\eta$: this is one of the features of the functionals defined in \Cref{def:penalty scalarizing functionals}. More precisely, fix $\rho > 1$ and define the function
\begin{equation}
    \label{eq:small lambda definition}
    \lambda_{p, \rho, \eta}:S \longrightarrow \R \quad\text{ s.t. }\quad \lambda_{\rho, \eta}(C) = s_\rho(\Lambda C - \eta) = -\|\Lambda C - \eta\|_p^p + \rho \|(\Lambda C - \eta)_+\|_p^p.
\end{equation}
Then any solution to the (classical) minimization problem
\begin{equation}
\label{prbl:vo minimization with SAF for lambda}
    \begin{split}
        & \text{minimize } \lambda_{p, \rho, \eta}(C) \\
        & \text{subject to } C \in S_L,
    \end{split}
\end{equation}
is Pareto optimal by \Cref{prop:Pareto sufficient conditions}. Notice that, as in \Cref{ch:Existence of solutions for the minimization problem}, we restricted our attention to the finite-dimensional space $H_L$ defined above (for some $L \geq 0)$. Since $S_L$ is compact, we see that also $\Lambda(S_L)$ is compact, and thus we have at least one solution by Weierstrass' theorem (the continuity of $\lambda_{\rho, \eta}$ following).
\begin{theorem}
\label{thm:Solution of the SAF problem}
     The restriction of problem \eqref{prbl:vo minimization with SAF for lambda} to the space $H_L$ has at least one solution.
\end{theorem}

Finally, we give a heuristic recipe about how to choose the hyperparameters $p$ and $\rho$.
\begin{heuristic}
\label{heur:heuristic on p}
If $p \rightarrow 0$, then $|f(\omega)|^p \rightarrow 1$ if $f(\omega) \neq 0$, and $|f(\omega)|^p \rightarrow 0$ if $f(\omega) = 0$. In particular, choosing a small value of $p$ can help prioritizing the minimization of the areas where $\Lambda C - \eta$ is positive but very small.  
\end{heuristic}
\begin{heuristic}
\label{heur:heuristic on rho}
Choosing a large value of $\rho$ prioritizes the positive term of $\lambda_{p, \rho, \eta}$, thus prioritizing the minimization of the error in the areas where $\Lambda C - \eta$ is positive.   
\end{heuristic}
We shall explore the above heuristics in more detail in \Cref{ch:Error profiles for varying p and rho}.

\subsection{Applicability of the target error profile hypothesis}
\label{ch:Applicability of the target error profile hypothesis}
In this section, we list some examples of real-world applications in which \Cref{hyp:ROI hypothesis - formal} seems plausible. The general philosophy is that a local approach is preferable whenever we know something about the frequency distribution of the class of signals we have to Fourier-transform. The list is obviously non-exhausting and serves only to establish the validity of the target error profile assumption in some cases.

\begin{itemize}
    \item In \cite{cao_locally_1995}, a technique to locally achieve higher spatial resolution in MRI is presented and discussed. In particular, areas where is necessary to resolve edges with higher resolution are considered, like the scalp or the skull of cranial spectroscopic images. A similar problem is addressed in \cite{layton_single_2013}, where $k$-space trajectories are optimized to improve resolution in a specific region of interest. MRI with non-cartesian $k$-spaces makes use of the adjoint NUFFT (usually iterating it several times \cite{fessler_nufft-based_2007}). Therefore, \Cref{hyp:ROI hypothesis - formal} can be applied.
    \item In \cite{schmid_nonequispaced_2016} NUFFT is applied to reconstruct the spectra of photoacoustic tomographic signals. In particular, the authors are interested in pre-selected region of interests of the spectra.
    \item In human eye OCT, users are particularly interested in the macular region. Pathological signs are commonly identified by analysing anomalies in retinal-layer segmentation, like drusen deposits \cite{liu_joint_2022}. Therefore, minimizing the reconstruction error in the vicinity of the retina can be useful.
    \item \Cref{hyp:ROI hypothesis - formal} can be assumed whenever one is interested in a reduced field of view, like in \cite{ye_optimal_2019}. In particular, if signals are real, one can focus only on half of the spectrum.
\end{itemize}

\begin{remark}
From the list above, it is evident that the choice of the parameters $\eta$ and $\rho$ is subjected to the application: choosing in which areas the accuracy should be higher constitutes \emph{a priori} information which in principle could not be available. In this case, \Cref{hyp:ROI hypothesis - formal} should not be assumed and the usage of a gridding kernel minimizing the \emph{global} error, like the first PSWF of order 0, should be considered preferable.     
\end{remark}

From the above discussion, it is evident that the definition of the target error profile $\eta$ and the penalty coefficient $\rho$ is quite subjective, and it must be crafted by hand depending on the application. The degree of accuracy that is needed in a certain area of the spectrum has to be decided \textit{a priori} by an external expert, usually called a \textit{decision maker}, and then translated into a target error by some procedure. This approach is known in literature as \textit{human-in-the-loop} (HILP) \cite{meignan_review_2015}, and any algorithm which includes steps with a human interaction is called an \emph{interactive method} \cite{miettinen_nonlinear_1998}. There are many of them, like trial and error, interactive reoptimization, trade-off based methods like Guided Multi-Objective Evolutionary Algorithm (G-MOEA). 

If for some reason HILP is not an available option, other approaches can be used to estimate the probability distribution of the spectra in the class of signals for the specific application. For example, the authors are going to propose a method based on neural networks in future work.  

\section{Numerical results}
\label{ch:Numerical results}
We compare the performance of our method with that of the most recent and most efficient implementation of the adjoint NUFFT on CPU, which (to the authors' knowledge) is FINUFFT, developed by A. Barnett and others from the Flatiron Institute \cite{barnett_parallel_2019}. Since this library does not let the user choose the gridding kernel to use (which is the main subject of the discussion above), we decided to develop our own adjoint NUFFT.
\begin{remark}
The optimization of our algorithm was carried out with the help of AI tools. More precisely, the code was first written by the authors, and then optimized by AI using common techniques (loop unrolling, cache locality, multi-threading, etc.). All the changes were doubly checked and did not involve the mathematics, but only the structure of the code.    
\end{remark}
The FINUFFT CPU library provides a multithread NUFFT algorithm that allows the user to choose a desired maximum error threshold and automatically selects the kernel support width $W$ accordingly. This threshold must be intended over the \emph{global} error. The originally chosen gridding kernel as the so-called "exponential of semicircle" (ES) kernel, which demonstrates superior performance with respect to previous implementations of the NUFFT like \cite{greengard_accelerating_2004, Fessler_minmaxNUFFT, keiner_nfft_nodate}, but recent disclosures by the authors seem to step back to the classical PSWF kernel \cite{flatiron_institute_improved_nodate}. 

To enhance fairness, we rebuilt FINUFFT to ensure compiler-optimized options. Nevertheless, we found that on our hardware our implementation of the multi-thread adjoint NUFFT is faster than FINUFFT even in the global case corresponding to $\eta(\omega) = 0$ for all $\omega \in \fov$. Since our main concern was to validate the above theory, we did not investigate this fact and we currently cannot say whether this apparent speed-up depends on the hardware or on the compilation settings or not.
\begin{remark}
    We tested our new method against both the FINUFFT implementation and a classical Kaiser-Bessel (KB) kernel (which is approximately optimal for the \emph{global} error). The first provides a comparison against the state-of-the-art algorithm. The second provides an internal comparison to check the validity of the theoretical discussion given in the previous sections. 
\end{remark}
All the code shall be available on a dedicated GitHub.

\subsection{Technical information}
\label{ch:Technical informations and limitations}
\Cref{thm:Solution of the SAF problem} gives us an explicit recipe to construct vector optimized gridding kernel bases on a specific target profile $\eta$: minimize the penalty functional $\lambda_{p, \rho, \eta}$ for a fixed choice of $p \in (0, +\infty)$ and $\rho$.
\subsubsection{Finite-dimensional Legendre approximation of the gridding kernels}
To effectively minimize $\lambda_{p, \rho, \eta}$, we need to approximate the gridding kernel $C$ with an element of a finite-dimensional vector. Although in the theory presented above $C$ was assumed to be an element of $\cont^2(\R)$, we decided not to use this regularity for numerical experiments, as it introduces unnecessary complications in the code (see \Cref{rmk:C does not need to be differentiable}). Instead, we decided to approximate $C$ by linear combination of the first $L + 1$ Legendre polynomials (truncated to have support in the interval $I_W$). The reasons why we adopted this basis are the following
\begin{itemize}
    \item it is orthogonal on the finite symmetric interval $[-1, 1]$ (and thus, on its dilation $[-W, W]$), thus making the computation of the coefficients of $C$ very easy,
    \item its element are easy to compute as well,
    \item its Fourier dual is analytically expressible in terms of spherical Bessel functions\cite{hale_algorithm_2014}, allowing to calculate $\ell_\nu(\omega, C)$ as in \eqref{eq:ell bound definition} without incurring in large numerical errors or heavy computations.
\end{itemize}
We found that this basis performs well even without modulating it by a bump function as in \eqref{eq:bump modulation}.

To effectively compute $\lambda_{p, \rho, \eta}$, the two sums in \eqref{eq:ell bound definition} are truncated to a suitable pre-selected $S_{\max{}}$. Since $\widehat{C}$ decays by definition, this operation can be done safely, provided that $S_{\max{}}$ is large enough. The same consideration holds for the Fourier coefficients of $\nu$.

\subsubsection{About the minimization of \texorpdfstring{$\lambda_{p, \rho, \eta}$}{lambda {p, rho, eta}}}
The minimization of $\lambda_{p, \rho, \eta}$ is performed using the \textit{trust-region} algorithm of the \textsc{Matlab} function \texttt{fminunc}, after suitable scaling. The initial guess is given by projecting the Kaiser-Bessel (KB) kernel onto the (complex) space spanned by the first $L + 1$ Legendre polynomials. Since $\lambda_{p,\rho,\eta}$ is a homogeneous function, the condition $C \in S_L$ can be relaxed to $C \neq 0$. This condition is enforced in the minimization by projecting out the redundant scaling direction and performing the optimization in the subspace orthogonal to the current iterate. Once the coefficients are obtained, we construct a (complex) cubic spline approximating the Legendre approximation of $C$. This cubic spline shall be effectively used in the adjoint NUFFT. The choice of the hyperparameters $p$ and $\rho$ is studied in details in \Cref{ch:Error profiles for varying p and rho}. We empirically found that setting $L = 15W$ is usually enough to get an accurate approximation of the gridding kernel by Legendre polynomials. 

\subsubsection{About the generation of the tested signals}
\label{ch:About the generation of the tested signals}
We test both our method and FINUFFT on the same sets composed of $100$ signals of varying dimensions $N = 10^3, 10^4, \dots, 10^8$ ($100$ signals for each dimension). The signals are linear combinations of complex harmonics constructed in the following way:
\begin{enumerate}
\itemsep-0.2em 
    \item we randomly select a number $Q$ of modes using the uniform distribution on $[1, 100] \cap \Z$,
    \item we randomly select $Q$ frequencies $\omega_1, \dots, \omega_Q$ using the uniform distribution on $\fov$,
    \item we randomly select $Q$ amplitudes $A_q$ using the uniform distribution on $[0, 1]$,
    \item we randomly select $N$ non-uniform time samples through the formula
    \begin{equation}
    \label{eq:times samples in experiments}
        t_n = n + \epsilon_n,
    \end{equation}
    where $\epsilon_n$ is distributed accordingly to $\nu$ (which is assumed to be given),
    \item we finally generate the signal $\bm{s} = (s_n) \in \C^N$ as the linear combination
    \begin{equation}
        s_n = \sum_{q = 1}^Q A_q \ifcoef{\omega_q t_n}. 
    \end{equation}
\end{enumerate}
In this way, for each $N$ we get a set $X_N$ of signals $\{s_1, \dots, s_{100}\}$ (notice that the value of $Q$ can vary inside each $X_N$).
\begin{remark}
    The actual distribution of the Fourier modes $\omega_q$, $q = 1, \dots, Q$ does not affect the final (local) accuracy of our method. In this sense, our method is completely independent of the distribution of the $\omega_q$. However, in real world applications, one shall obviously design $\eta$ to enhance accuracy in the areas where it is more likely to find $\omega_q$.  
\end{remark}

\subsection{Limitations of the method}
The method described in this work is naturally subject to some limitations, the most evident is likely the impossibility to choose \textit{a priori} the accuracy of the NUFFT. The FINUFFT implementation lets the user choose the error tolerance, and then it automatically computes the kernel width that best suits it. We cannot claim the same for the method proposed in this paper: we are currently unable to determine which level of accuracy is reached by a kernel computed through vector optimization \textit{before} actually computing it. Users need to check it by trial-and-error for several $W$ and choose the one that best suits, although some heuristics can be made to simplify the problem. However, this does not really constitute a big issue, since optimization usually terminates very fast.    

Another major limitation of the method is due to round-off errors. Regardless of the chosen gridding kernel, we are currently unable to reach an absolute error less than about $10^{-13}$, which is larger than machine precision of about $100/200$ times. This level of accuracy corresponds to that reached over the global spectrum by the Kaiser-Bessel kernel with parameter $W=7$. However, the same barrier limits the accuracy of the FINUFFT algorithm (see \Cref{fig:MAE global}).

\begin{figure}[h]
    \centering
    \includegraphics[width=0.5\linewidth]{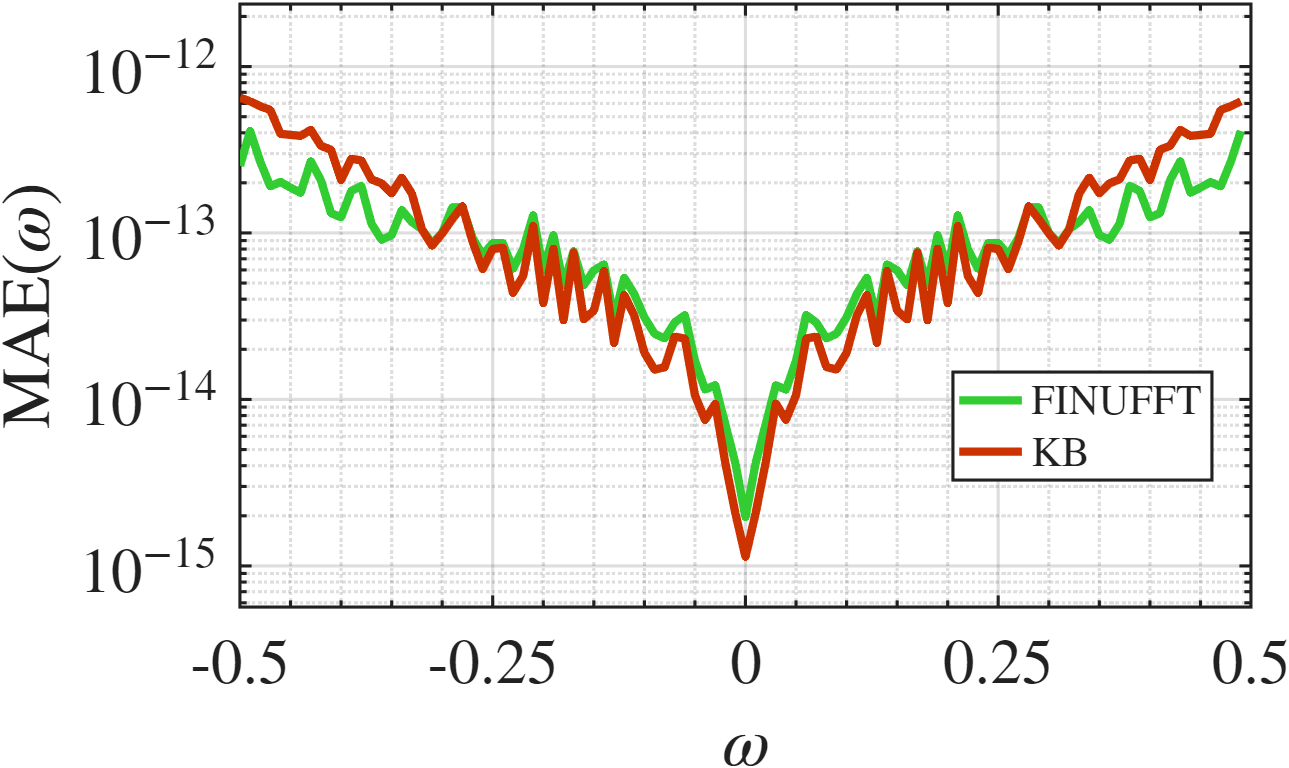}
    \caption{The minimum absolute \textit{global} error that can be achieved by both FINUFFT and our implementation of the adjoint NUFFT (using the KB kernel).}
    \label{fig:MAE global}
\end{figure}

As $W$ increases, we need more subintervals to approximate the gridding kernel through a cubic spline with the needed accuracy. The number $L+1$ of Legendre polynomials must increase as well. Although this facts do not affect the time performance in a sensible way, it may constitute a problem for scenarios with low memory capacity. We found that using $L = 15 W$ and $2WD$ subintervals, with $D = 300$, suffices even for the best-approximation case corresponding to $W = 7$, but smaller values of $D$ can be used for smaller $W$. 

The exact choice of $p$ and $\rho$, as well as $\eta$, should be considered application-dependent. In general, following \Cref{heur:heuristic on p} and \Cref{heur:heuristic on rho}, small values of $p$ and high values of $\rho$ should be used to prioritize the minimization in the areas in which $\eta$ is smaller. Of course, using $p \leq 1$ (or very high values of $\rho$) ultimately causes the penalty scalarizing functional to be non-differentiable, which could constitute a problem for the chosen minimization algorithm.

Finally, it is worth mentioning that we can make no statement of Pareto-optimality for the gridding kernels numerically found by minimizing $\lambda_{p, \rho, \eta}$ for a certain choice of $p$, $\rho$ and $\eta$. For $p \leq 1$ the problem is not even differentiable, and thus we cannot even claim that the solutions are \textit{minima} of $\lambda_{p, \rho, \eta}$. Nevertheless, in all the cases we considered, the optimization produces new gridding kernels that substantially improved the accuracy in the prescribed ROI, thus confirming the ideas of the theory presented above. 

\subsection{Tests discussion}
In this section, we present two kind of numerical tests with two different purposes: the firsts are performed to numerically validate the theory presented above, while the seconds are meant to further clarify the meaning of the hyperparameters $p$ and $\rho$.

\subsubsection{Testing the accuracy for a prescribed \texorpdfstring{$\eta$}{eta}}
For these tests, we consider the three different target error profiles $\eta_1$, $\eta_2$, and $\eta_3$, shown in \Cref{fig:eta}. In the first, we are interested in maximizing the accuracy in the positive half of the spectrum, in the second, we focus on a smaller interval centred around $\omega_0 = 0.25$, and in the third, on a ROI composed by two intervals of different lengths. The third is different from the other two because the two intervals corresponds to two different target errors ($10^{-25}$ and $10^{-27}$), and not only in the set $\{0, 1\}$ (i.e. it is not a characteristic function). In this way, we prove that we can reach different levels of accuracy across the spectrum, as explained in \Cref{ch:Existence of Pareto optimal solutions with prescribed shapes}. We set the oversampling factor $\gamma = 2$, which is the same value used by FINUFFT. 
\begin{figure}[h]
    \centering
    \begin{subfigure}{0.49\textwidth}
    \includegraphics[width=\textwidth]{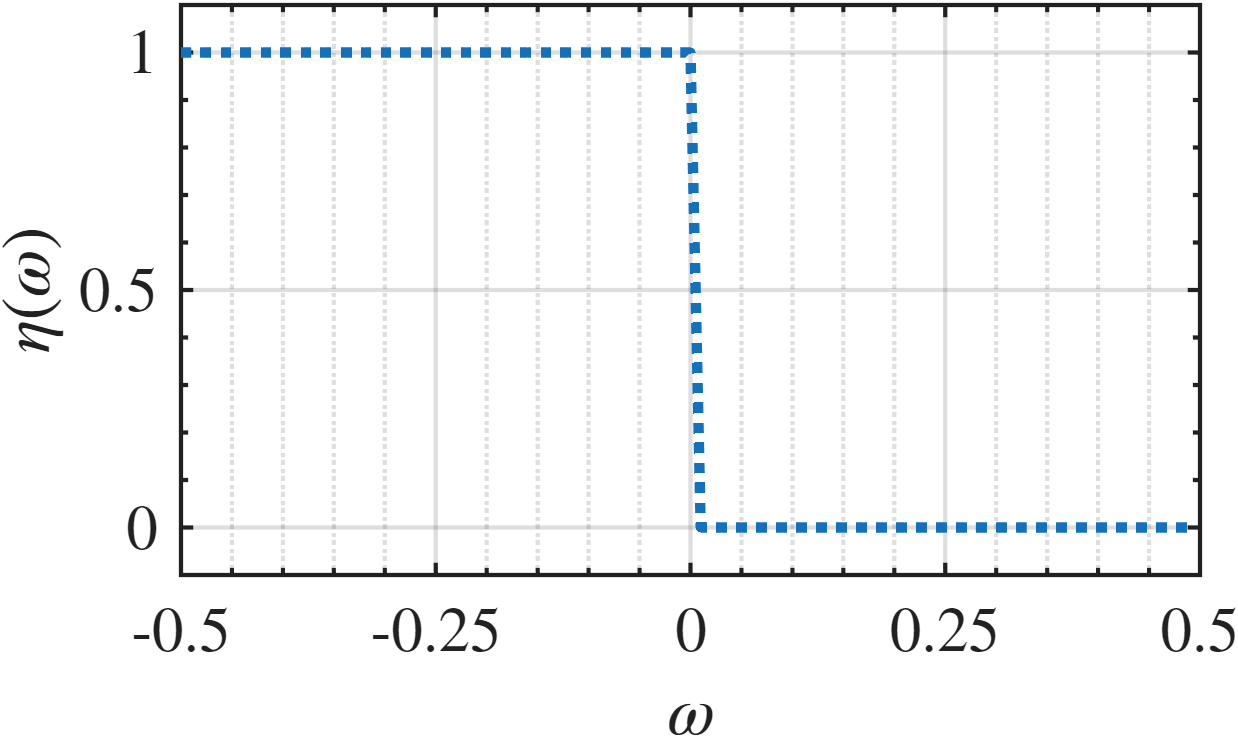}
    \caption{Target error profile $\eta_1$.}
    \end{subfigure}
    \hfill
    \begin{subfigure}{0.49\textwidth}
    \includegraphics[width=\textwidth]{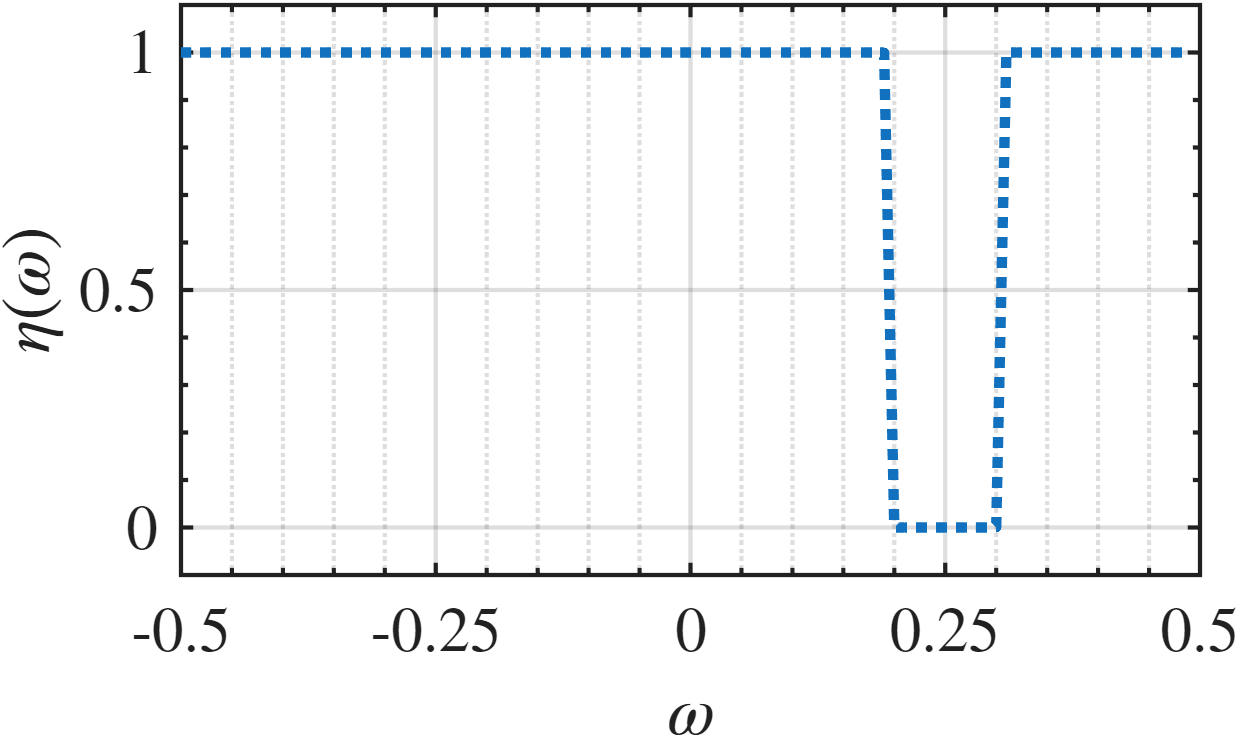}
    \caption{Target error profile $\eta_2$.}
    \end{subfigure}
    \begin{subfigure}{0.49\textwidth}
    \includegraphics[width=\textwidth]{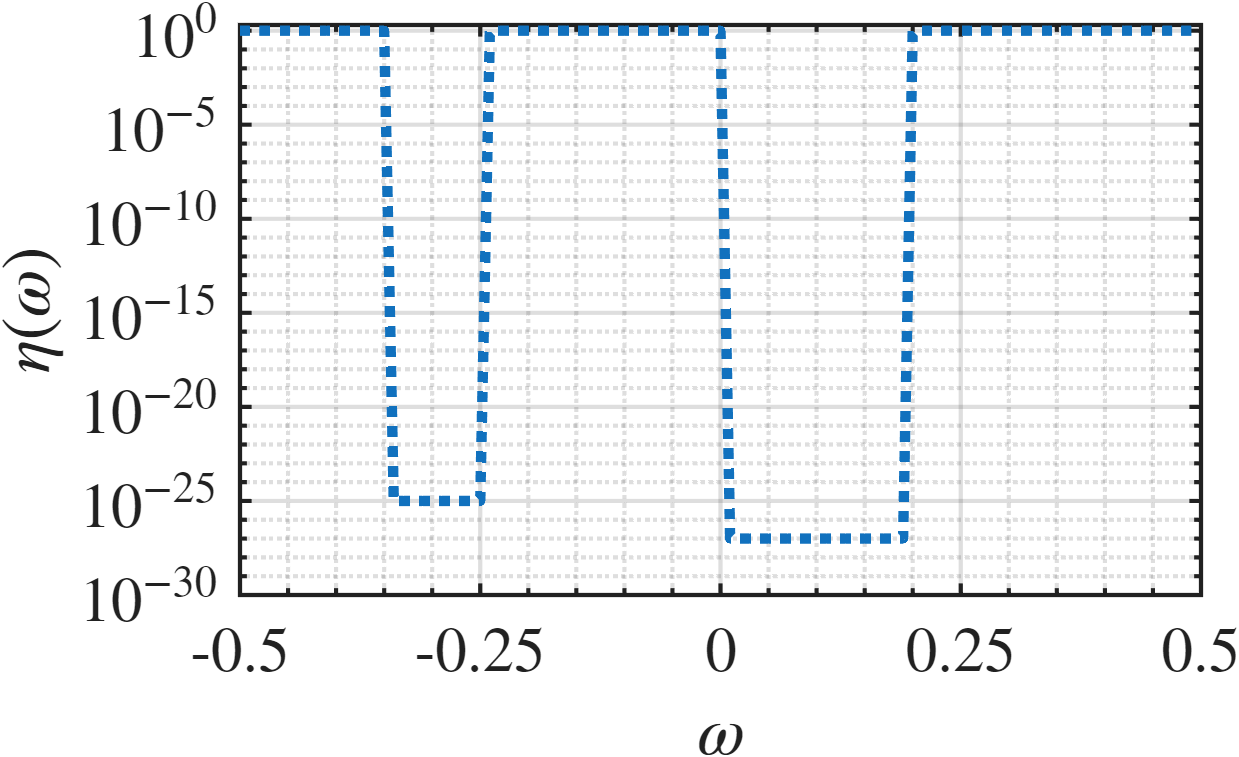}
    \caption{Target error profile $\eta_3$.}
    \end{subfigure}
    \caption{The three target error profiles $\eta$ used in the tests: $\eta_1$ maximizes the accuracy on $[0, \frac{1}{2})$, $\eta_2$ on $[0.2, 0.3]$, and $\eta_3$ on $[-0.35, -0.25] \cup [0, 0.2]$ (with different levels on each subinterval).}
    \label{fig:eta}
\end{figure}
Since we are most interested in showing how we can design gridding kernels with higher levels of accuracy in a prescribed ROI, we chose a small $p$ and a large $\rho$, so that they prioritize the positive part of $\lambda_{p, \rho, \eta}$, which is responsible for decreasing the error in the areas where $\eta$ is very small (see \Cref{heur:heuristic on p} and \Cref{heur:heuristic on rho}). More specifically, for these tests we set $p = 0.5$ and $\rho = 10^2 \cdot A/B$, where $A = \|\Lambda C_0 - \eta\|_p^p$, $B = \|(\Lambda C_0 - \eta)_+\|_p^p$, and $C_0$ is the KB kernel. All the parameters are summarized in \Cref{tab:parameters eta test}.
\begin{remark}
The choice of $p$ and $\rho$ is application-dependent, thus we can give no exact technique for selecting them. Users can choose them by trial-and-error methodologies, with the help of \Cref{heur:heuristic on p} and \Cref{heur:heuristic on rho}. This is generally not a big problem, since optimizing $\lambda_{p, \rho, \eta}$ is usually quite fast. 
\end{remark}
We test our method for $W = 1, 2, 3, 4, 5, 6$ ($W = 7$ has been excluded because it corresponds to the minimum possible error achievable by both our method and FINUFFT), and $\nu$ equal to the Lebesgue measure on $[-\frac{1}{2}, \frac{1}{2}]$. This choice of $\nu$ is the most natural and straightforward, but we tested other options in \Cref{sec:Testing the accuracy for a prescribed nu}. The comparisons are done in this way: for each $\eta$ and $W$, we first compute the vector-optimized gridding kernel (VO) which maximizes the accuracy in the selected ROI. We then recompute the NUFFT of the signals using the FINUFFT algorithm, the classical KB kernel implemented in our algorithm, and the new VO kernel, so that all three methods reach at least the same accuracy on the ROI (see \Cref{fig:MAE 3 W3} for the example case $W = 3$). The KB kernel typically requires a gridding kernel width $W' = W +1$ to reach the same level of accuracy, but if the ROI is particularly small, as for the second target error profile, we can have $W' = W + 2$. In \Cref{fig:MAE 3 W3} is also reported the plot of the mean absolute error (MAE), depending on $\omega$, returned by the KB kernel with the same width of the VO kernel. The MAE is computed as the difference between the NUFFT of the signals and their NUDFT, computed using a (slow) matrix-vector multiplication. In practice, for each $\omega$, the MAE at $\omega$ is just the square root of the error $E(\bm{u}; \omega, h, C)$ defined in \eqref{eq:error}, averaged over the signals. From \Cref{fig:MAE 3 W3}, we can see that the accuracy level of the VO kernel is noticeably improved. 
\begin{remark}
    Following the advice given in \cite{barnett_parallel_2019}, we keep the output dimension of the NUFFT algorithm smaller than the input dimension (we used an output dimension of $10^2$ for all $N$). The reason is that we are mainly interested in testing the performance of the gridding step of the NUFFT, rather than that of the employed FFT. 
\end{remark}
\begin{table}[h]
\label{tab:eta tests parameters}
    \centering
    \small
    {\renewcommand{\arraystretch}{1.2}
    \begin{NiceTabular}{|c|c||c|c|}
    \hline
        $D$                 & $300$ & $L$ & $15 W$ \\
        $\gamma$            & $2$ & $p$ & $0.5$ \\
        $\rho$              & $10^2 \cdot \|\Lambda C_0 - \eta\|_p^p / \|(\Lambda C_0 - \eta)_+\|_p^p$ & $\nu$ & Lebesgue measure on $\fov$ \\
        $S_{\max{}}$        & $64$ & output dimension    & $100$ \\
        \hline
    \end{NiceTabular}}
    \caption{The parameters chosen to test kernels found using the VO technique discussed above for varying $W$ and $N$.}
    \label{tab:parameters eta test}
\end{table}
In \Cref{tab:times varying W eta_1}, \Cref{tab:times varying W eta_2}, and \Cref{tab:times varying W eta_3}, the computational times in the case $N=10^7$ are reported for the FINUFFT implementation, our implementation with the global-optimal kernel KB, and our implementation with the VO kernel. In the last column, we reported the times returned by the KB kernel having the same width of the VO kernel (and thus suboptimal) as a base case. All algorithms are tested in both the single thread and the multi-thread ($12$ threads) mode. We can see that the base KB and the VO kernel take the same amount of time, while the KB kernel with width adjusted to reach the same level of accuracy of VO is sensibly slower. This is exactly what we expected: to reach the same accuracy in the ROI using a KB kernel, we need larger $W$ and thus more time. Naturally, the time of KB and VO becomes closer and closer as $W$ increases, since the ratio $W / W + 1$ is closer to $1$: for example, in the multi-thread case of \Cref{tab:times varying W eta_2}, the VO time is about $66\%$ of the KB one for $W = 1$, while for $W = 6$ it is just $84\%$. 

On the other side, the performance of VO against the FINUFFT implementation seems to depend on the chosen mode. Using the multi-thread modality, our implementation regularly outperforms FINUFFT for any test and any $W$, while in the single thread one the faster depends on $W$ (VO wins for smaller $W$, FINUFFT for larger one). Although we cannot explain this behaviour, since we do not know the details of the FINUFFT implementation, also in this case we suspect that the difference between the two modalities is given by non-optimal handling of the cache in our implementation, whose effects are more visible when the number of threads is small (we saw that our implementation beats FINUFFT starting from $4$ threads). Anyway, it is important to recall that FINUFFT is a highly optimized library, while we did not spend much time in the optimization of the code (which was carried on with the help of some AI tools).
\begin{table}[h]
    \centering
    \small
    \begin{NiceTabular}{|c||c|c|c:c|}
         \hline
         & \multicolumn{4}{c}{$\eta_1$ - (Single thread)} \\
         \hline
         $W$ & FINUFFT & KB ($W+1$) & VO & KB ($W$) \\
         \hline
         $1$ & $0.18$ & $0.14$ & $\bm{0.10}$ & $0.10$ \\
         \hline
         $2$ & $0.19$ & $0.22$ & $\bm{0.13}$ & $0.13$ \\
         \hline
         $3$ & $\bm{0.21}$ & $0.24$ & $\bm{0.21}$ & $0.21$ \\
         \hline
         $4$ & $\bm{0.24}$ & $0.30$ & $0.27$ & $0.27$ \\
         \hline
         $5$ & $\bm{0.25}$ & $0.31$ & $0.27$ & $0.27$ \\
         \hline
         $6$ & $\bm{0.24}$ & $0.39$ & $0.33$ & $0.36$ \\
         \hline
    \end{NiceTabular}
    \hfill
        \begin{NiceTabular}{|c||c|c|c:c|}
         \hline
         & \multicolumn{4}{c}{$\eta_1$ - (Multi-thread)} \\
         \hline
         $W$ & FINUFFT & KB ($W+1$) & VO & KB ($W$) \\
         \hline
         $1$ & $0.084$ & $0.034$ & $\bm{0.021}$ & $0.024$ \\
         \hline
         $2$ & $0.088$ & $0.041$ & $\bm{0.034}$ & $0.034$ \\
         \hline
         $3$ & $0.092$ & $0.052$ & $\bm{0.044}$ & $0.049$ \\
         \hline
         $4$ & $0.100$ & $0.064$ & $\bm{0.054}$ & $0.057$ \\
         \hline
         $5$ & $0.099$ & $0.070$ & $\bm{0.061}$ & $0.066$ \\
         \hline
         $6$ & $0.117$ & $0.085$ & $\bm{0.072}$ & $0.076$ \\
         \hline
    \end{NiceTabular}
    \caption{Computational times (on average) for the first target error profile, with $N = 10^7$ and varying $W$. On the left are shown the single-thread times, on the right the multi-thread ones ($12$ threads).}
    \label{tab:times varying W eta_1}
\end{table}
\begin{table}[h]
    \centering
    \small
       \begin{NiceTabular}{|c||c|c|c:c|}
         \hline
         & \multicolumn{4}{c}{$\eta_2$ - (Single thread)} \\
         \hline
         $W$ & FINUFFT & KB ($W+2$) & VO & KB ($W$) \\
         \hline
         $1$ & $0.18$ & $0.13^*$ & $\bm{0.09}$ & $0.10$ \\
         \hline
         $2$ & $0.20$ & $0.22$ & $\bm{0.14}$ & $0.14$ \\
         \hline
         $3$ & $\bm{0.23}$ & $0.28$ & $0.22$ & $0.23$ \\
         \hline
         $4$ & $\bm{0.23}$ & $0.31$ & $0.25$ & $0.26$ \\
         \hline
         $5$ & $\bm{0.24}$ & $0.33^*$ & $0.28$ & $0.29$ \\
         \hline
         $6$ & $\bm{0.25}$ & $0.36^*$ & $0.33$ & $0.36$ \\
         \hline
    \end{NiceTabular}
    \hfill
        \begin{NiceTabular}{|c||c|c|c:c|}
         \hline
         & \multicolumn{4}{c}{$\eta_2$ - (Multi-thread)} \\
         \hline
         $W$ & FINUFFT & KB ($W+2$) & VO & KB ($W$) \\
         \hline
         $1$ & $0.076$ & $0.030^*$ & $\bm{0.021}$ & $0.021$ \\
         \hline
         $2$ & $0.079$ & $0.046$ & $\bm{0.030}$ & $0.030$ \\
         \hline
         $3$ & $0.092$ & $0.061$ & $\bm{0.043}$ & $0.043$ \\
         \hline
         $4$ & $0.093$ & $0.067$ & $\bm{0.050}$ & $0.051$ \\
         \hline
         $5$ & $0.109$ & $0.074^*$ & $\bm{0.063}$ & $0.061$ \\
         \hline
         $6$ & $0.133$ & $0.099^*$ & $\bm{0.084}$ & $0.086$ \\
         \hline
    \end{NiceTabular}
    \caption{Computational times (on average) for the second target error profile, with $N = 10^7$ and varying $W$. On the left are shown the single-thread times, on the right the multi-thread ones ($12$ threads). \\
    \small * A kernel width $W' = W + 1$ was used.}
    \label{tab:times varying W eta_2}
\end{table}
\begin{table}[h]
    \centering
    \small
       \begin{NiceTabular}{|c||c|c|c:c|}
         \hline
         & \multicolumn{4}{c}{$\eta_3$ - (Single thread)} \\
         \hline
         $W$ & FINUFFT & KB ($W+1$) & VO & KB ($W$) \\
         \hline
         $1$ & $0.11$ & $0.11$ & $\bm{0.09}$ & $0.10$ \\
         \hline
         $2$ & $0.20$ & $0.19$ & $\bm{0.11}$ & $0.11$ \\
         \hline
         $3$ & $0.21$ & $0.22$ & $\bm{0.19}$ & $0.20$ \\
         \hline
         $4$ & $0.23$ & $0.27$ & $\bm{0.22}$ & $0.22$ \\
         \hline
         $5$ & $\bm{0.24}$ & $0.29$ & $0.26$ & $0.27$ \\
         \hline
         $6$ & $\bm{0.23}$ & $0.30$ & $0.27$ & $0.27$ \\
         \hline
    \end{NiceTabular}
    \hfill
        \begin{NiceTabular}{|c||c|c|c:c|}
         \hline
         & \multicolumn{4}{c}{$\eta_3$ - (Multi-thread)} \\
         \hline
         $W$ & FINUFFT & KB ($W+1$) & VO & KB ($W$) \\
         \hline
         $1$ & $0.075$ & $0.031$ & $\bm{0.021}$ & $0.020$ \\
         \hline
         $2$ & $0.078$ & $0.039$ & $\bm{0.030}$ & $0.031$ \\
         \hline
         $3$ & $0.091$ & $0.054$ & $\bm{0.045}$ & $0.045$ \\
         \hline
         $4$ & $0.100$ & $0.062$ & $\bm{0.052}$ & $0.057$ \\
         \hline
         $5$ & $0.104$ & $0.076$ & $\bm{0.064}$ & $0.066$ \\
         \hline
         $6$ & $0.112$ & $0.093$ & $\bm{0.078}$ & $0.078$ \\
         \hline
    \end{NiceTabular}
    \caption{Computational times (on average) for the third target error profile, with $N = 10^7$ and varying $W$. On the left are shown the single-thread times, on the right the multi-thread ones ($12$ threads).}
    \label{tab:times varying W eta_3}
\end{table}
\begin{figure}[h!]
    \centering
    \begin{subfigure}{0.49\textwidth}
    \includegraphics[width=\textwidth]{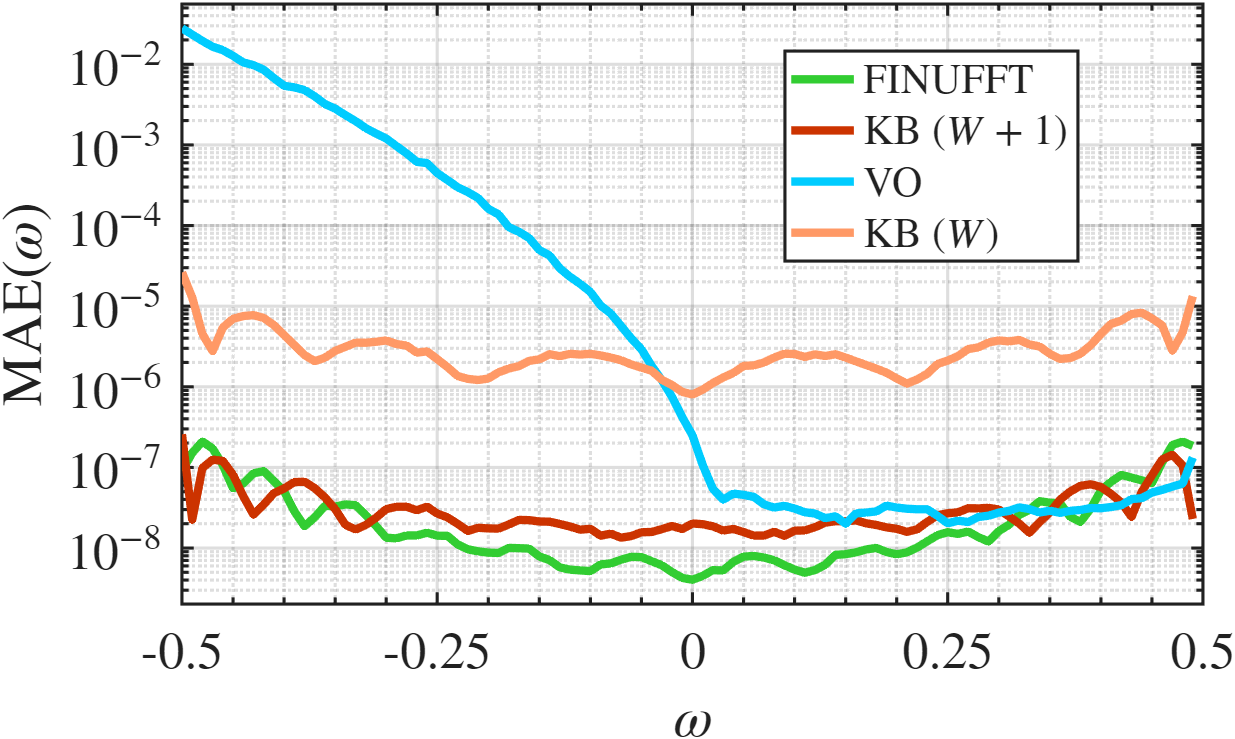}
    \end{subfigure}
    \hfill
    \begin{subfigure}{0.49\textwidth}
    \includegraphics[width=\textwidth]{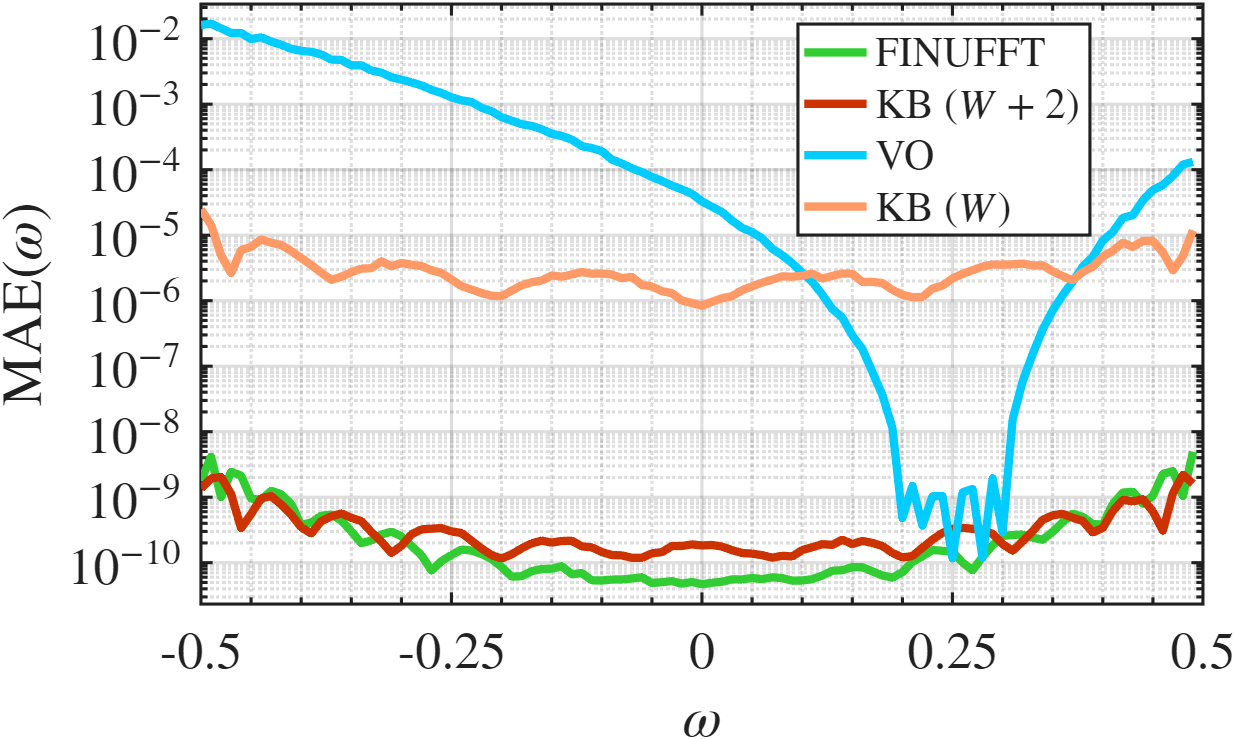}
    \end{subfigure}
    \begin{subfigure}{0.49\textwidth}
    \includegraphics[width=\textwidth]{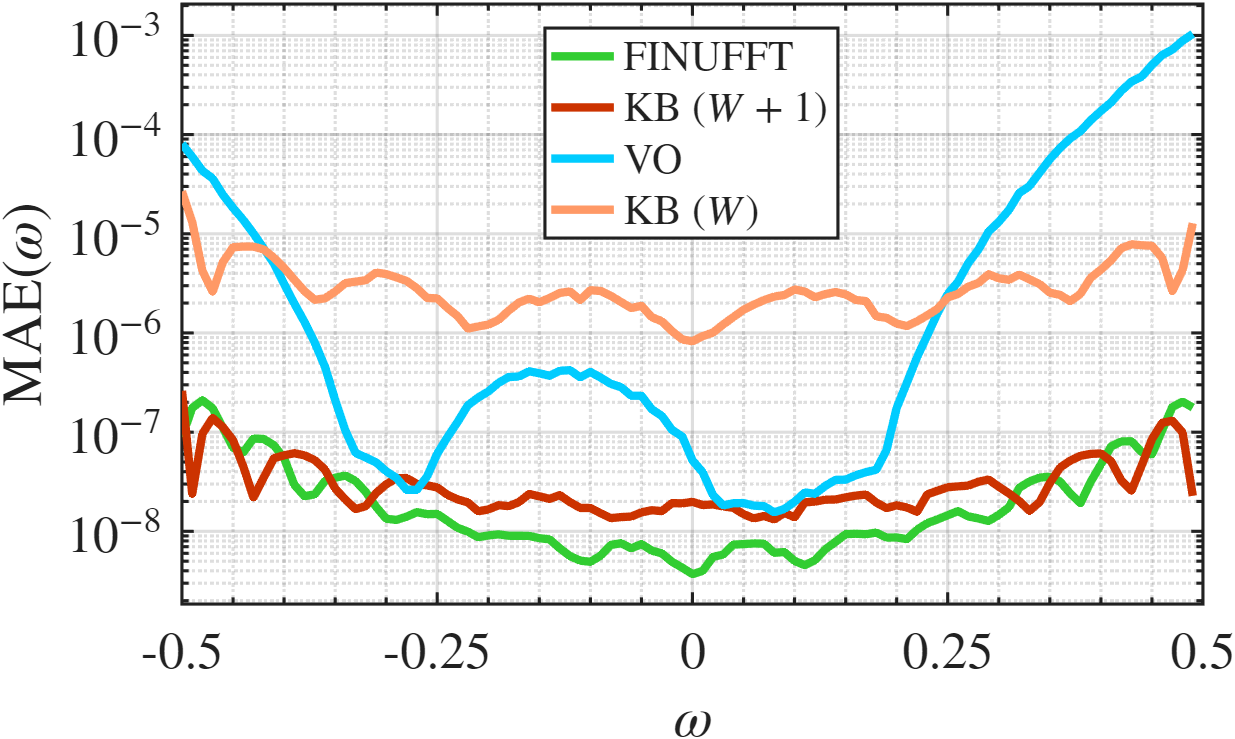}
    \end{subfigure}
    \caption{MAEs returned by FINUFFT (green), KB kernel (red), and VO kernel (blue) for the selected $\eta_1$, $\eta_2$, and $\eta_3$, respectively. The KB kernel and the FINUFFT tolerance are chosen so that the accuracy is the same as that returned by the VO kernel in the ROI. The orange line represents the MAE returned by the KB kernel having the same width $W$ of the VO kernel.}
    \label{fig:MAE 3 W3}
\end{figure}
Finally, in \Cref{fig:times} the average computational times of FINUFFT, KB kernel, and VO kernel, are shown for fixed $W = 3$ and varying $N$ (multi-thread modality with $12$ threads). We can see that for $N \geq 10^6$ the VO kernel is faster than the KB kernel and the FINUFFT. The VO time is about the $80 \%$, $70 \%$, and $75\%$, of the KB kernel respectively. These percentages seem to be related to the total width of the ROI. Similar observations can be made for different values of $W$.
\begin{figure}[h]
    \centering
    \begin{subfigure}{0.49\textwidth}
    \includegraphics[width=\textwidth]{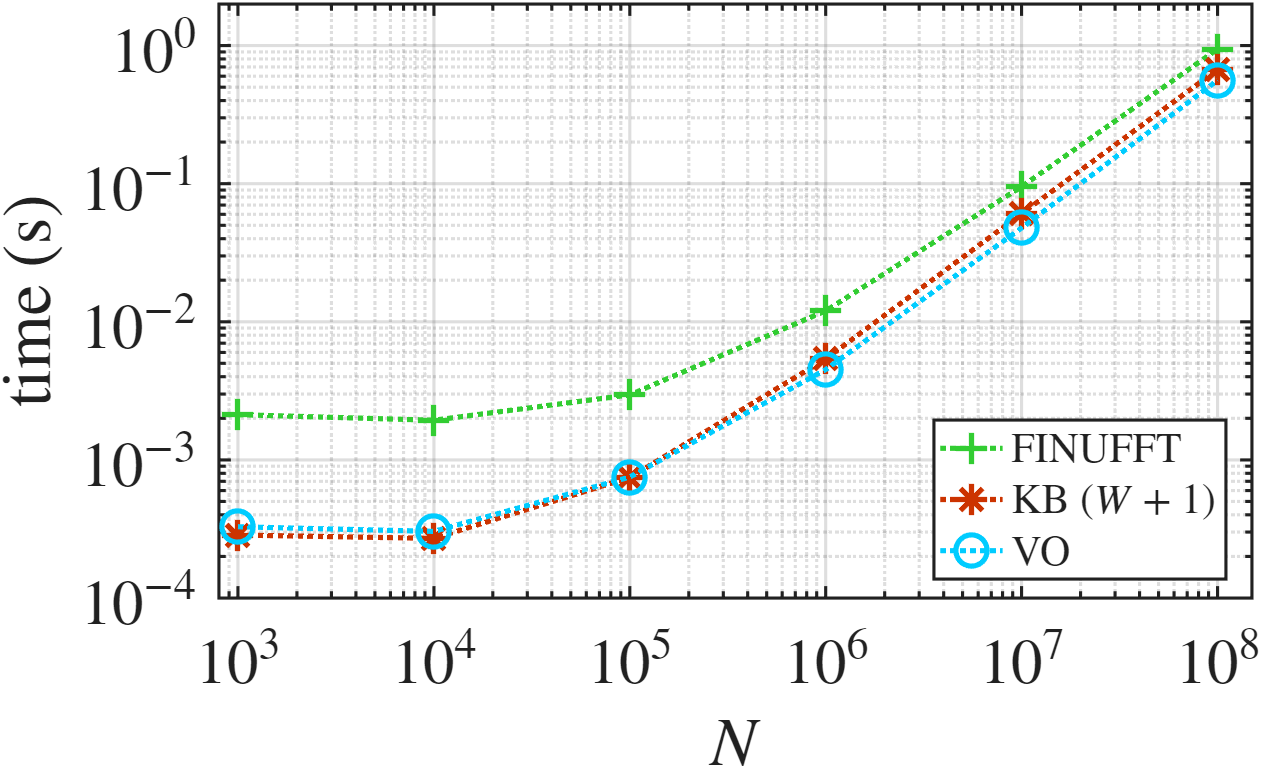}
    \end{subfigure}
    \hfill
    \begin{subfigure}{0.49\textwidth}
    \includegraphics[width=\textwidth]{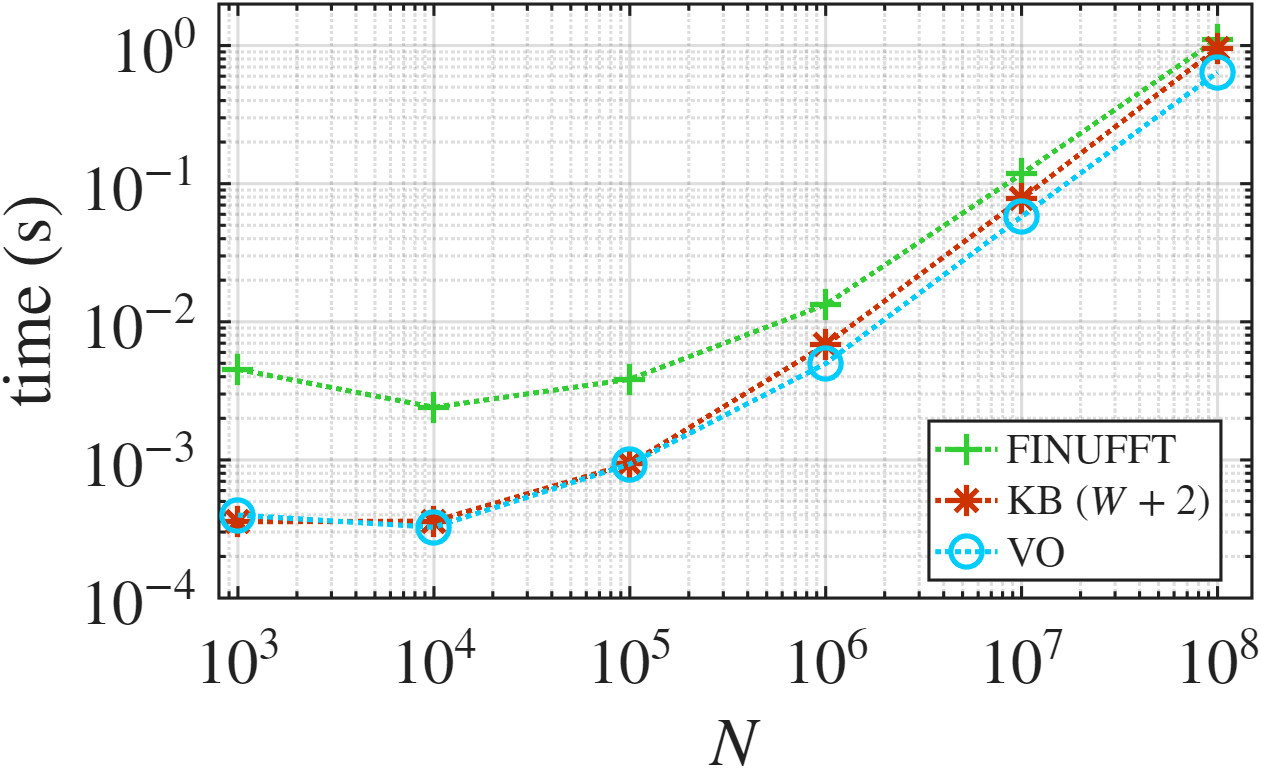}
    \end{subfigure}
    \begin{subfigure}{0.49\textwidth}
    \includegraphics[width=\textwidth]{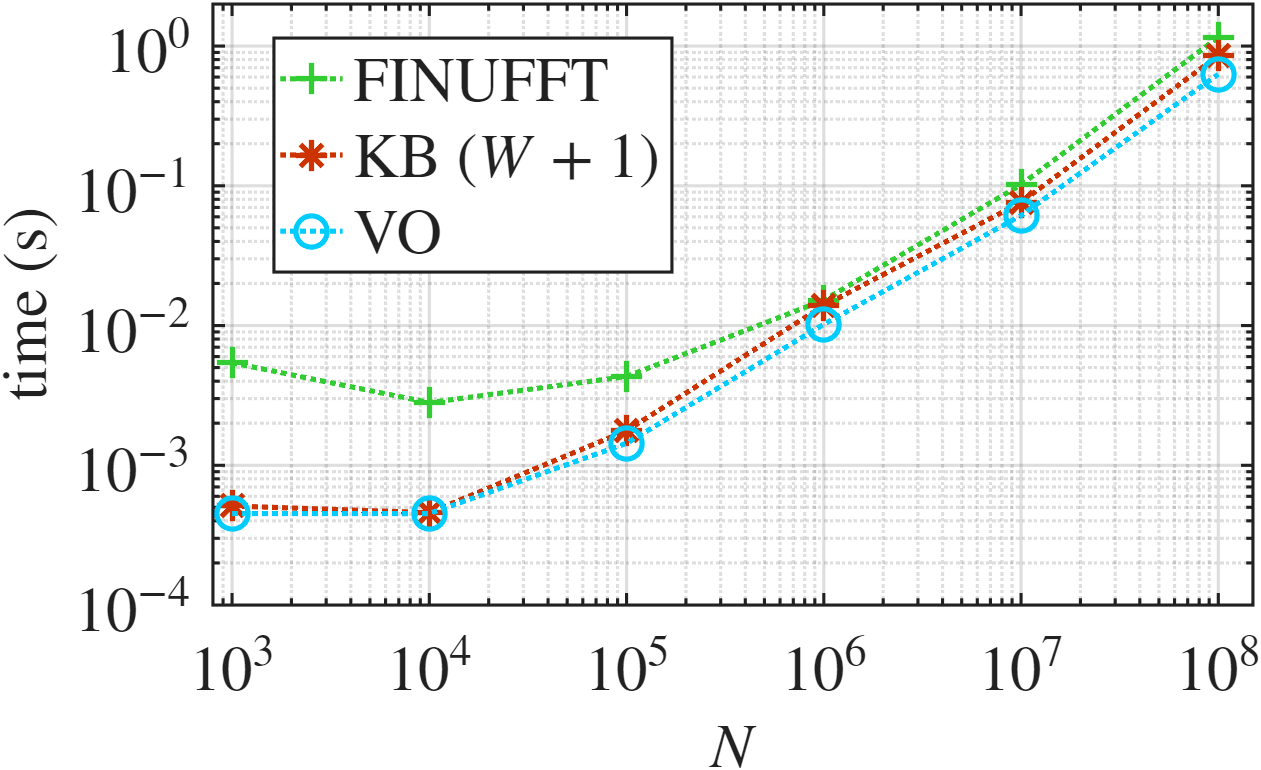}
    \end{subfigure}
    \caption{Computational times (on average) in seconds of FINUFFT (green), KB kernel (red), and VO kernel (blue), for varying $N$ in the case $W = 3$ (multi-thread modality, $12$ threads).}
    \label{fig:times}
\end{figure}

\subsubsection{Testing the accuracy for a prescribed \texorpdfstring{$\nu$}{nu}}
\label{sec:Testing the accuracy for a prescribed nu}
In this section, we show how, assuming \Cref{hyp:time samples are equidistributed - formal}, we can generate gridding kernels that are capable of providing a higher (even global) accuracy than the classical KB kernel, provided that $\nu$ is not the Lebesgue measure on $\fov$. For these tests, we do not use FINUFFT, since its tolerance is decided \textit{a priori}. Instead, we show how, using the same $W$, the classical KB kernel is outperformed by the VO kernel found by minimizing $\lambda_{p, \rho, \eta_0, \nu}$, where $\eta_0(\omega) = 0$ for all $\omega \in \fov$, and $\nu$ is the measure, representing the sampling as in \Cref{hyp:time samples are equidistributed - formal}, which is used in \eqref{eq:times samples in experiments}.

For this test, we used the family of Gaussian measures $\nu_\sigma$, parametrized by $\sigma \in(0,+\infty)$, whose Radon-Nikodym derivative with respect to the Lebesgue measure is 
\begin{equation}
    \frac{d\nu_\sigma}{d\tau}(\tau) = \frac{1}{\sqrt{2\pi \sigma^2}}e^{-\frac{\tau^2}{2 \sigma ^2}}, \quad \forall \tau \in \R. 
\end{equation}
The idea behind these measures is that, as $\sigma \rightarrow 0$, they become more and more concentrated around $\tau_0 = 0$, with the ideal limit case $\sigma = 0$ where $\nu_\sigma$ becomes a Dirac delta representing a perfectly uniform sampling, that is, a classical FFT suffices to compute the Fourier transform of the signals. The KB kernel is optimal for the case $\sigma \rightarrow +\infty$ corresponding to the limit case where $(\gamma\nu_\sigma\mod{1})$ is the uniform measure on $\fov$. For this reason, we expect the KB kernel to be suboptimal for small values of $\sigma$, and thus the VO kernel to reach a higher global level of accuracy, while keeping the same $W$ and thus the same computational time, provided that $\nu_\sigma$ is used at step $4$ of \Cref{ch:About the generation of the tested signals}. The parameters of this test are listed in \Cref{tab:parameters nu test}. 

\begin{table}[h]
    \centering
    \small
    {\renewcommand{\arraystretch}{1.2}
    \begin{NiceTabular}{|c|c||c|c||c|c|}
    \hline
        $D$ & $300$ & $L$    & $15 W$ & $\gamma$     & $2$ \\
        $p$ & $0.5$ & $\rho$ & $1.5$  & $S_{\max{}}$ & $64$ \\ 
        output dimension & $100$ & $W$ & $3$ & $N$ & $10^7$ \\
        \hline
    \end{NiceTabular}}
    \caption{The parameters chosen to test kernels found using the VO technique discussed above for varying $\nu_\sigma$.}
    \label{tab:parameters nu test}
\end{table}

In \Cref{fig:MAE nus} the MAEs are shown for different values of $\sigma$ ranging from $0.001$ to $0.5$. As we can see, the errors generally decrease as $\sigma \rightarrow 0$ and the measure $\nu_\sigma$ becomes more and more different from the Lebesgue one, although some oscillations occur. This result is also confirmed by \Cref{tab:mean errors nus}, where the average values of the MAE in $\fov$ are shown. Noticeably, the computational times of KB and VO are the same.  

\begin{figure}[h!]
    \centering
    \includegraphics[width=0.5\linewidth]{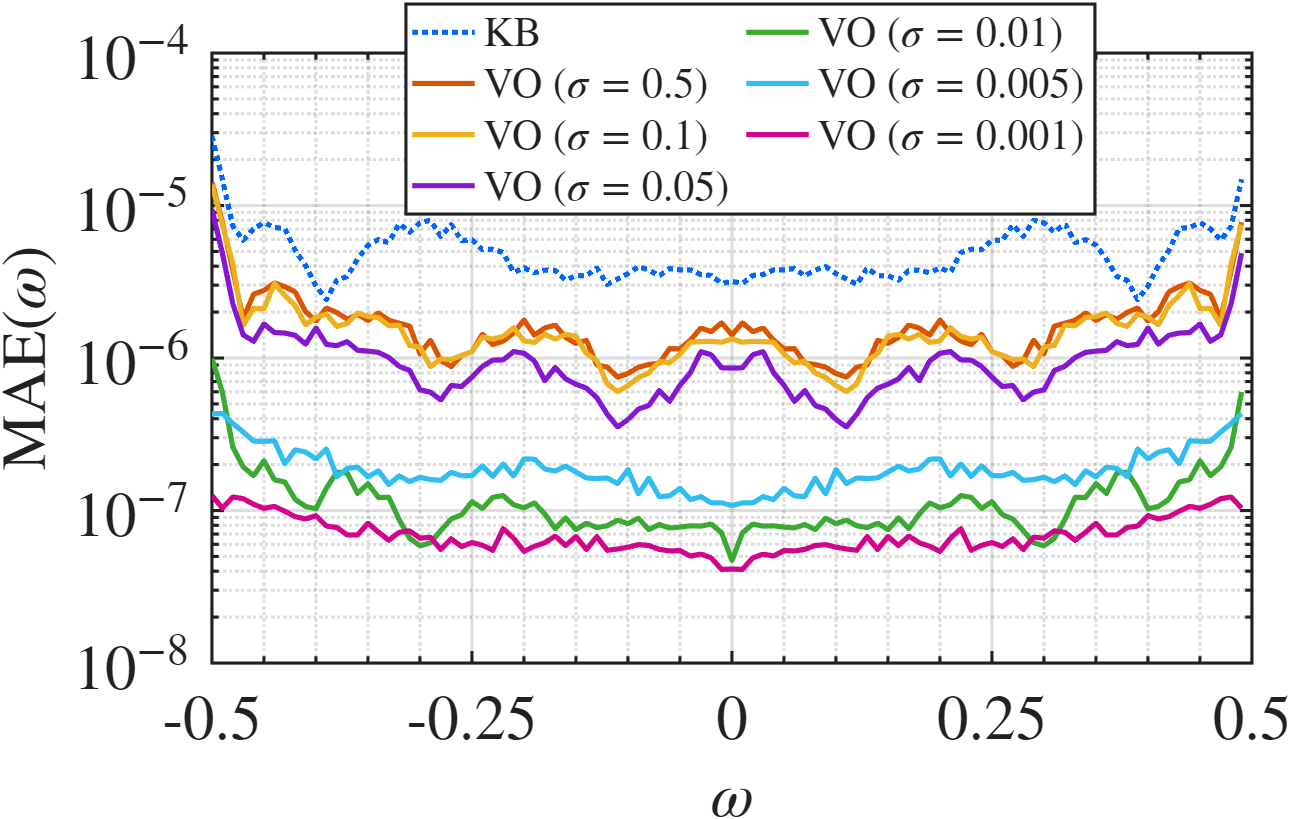}
    \caption{The MAEs for varying $\sigma$, with $W = 3$ and $N = 10^7$ fixed.}
    \label{fig:MAE nus}
\end{figure}

\begin{table}[h!]
    \centering
    \small
    {\renewcommand{\arraystretch}{1.2}
    \begin{NiceTabular}{|c|c|c|c|c|c|c|}
        \hline
       KB & VO ($\sigma = 0.5$) & VO ($\sigma = 0.1$) & VO ($\sigma = 0.05$) & VO ($\sigma = 0.01$) & VO ($\sigma = 0.005$) & VO ($\sigma = 0.001$) \\
       \hline
       $3.40\cdot 10^{-6}$ & $1.84\cdot 10^{-6}$ & $1.71\cdot 10^{-6}$ & $1.10\cdot 10^{-6}$ & $1.28\cdot 10^{-7}$ & $1.93\cdot 10^{-7}$ & $7.08\cdot 10^{-8}$ \\ 
       \hline
    \end{NiceTabular}}
    \caption{The average MAE in $\fov$ for the KB kernel and the VO kernel obtained using different decreasing values of $\sigma$ ($W = 3$, $N = 10^7$).}
    \label{tab:mean errors nus}
\end{table}

\subsubsection{Error profiles for varying \texorpdfstring{$p$}{p} and \texorpdfstring{$\rho$}{rho}}
\label{ch:Error profiles for varying p and rho}
We finally discuss in more detail \Cref{heur:heuristic on p} and \Cref{heur:heuristic on rho} by giving numerical evidence of their validity. We used the target error profile $\eta_2$ (see \Cref{fig:eta}) because it is very localized and thus it makes it easy to recognize if the MAEs are sharp or not around the ROI, that is the interval $[0.2, 0.3]$. We also set $W = 3$, $N = 10^5$, and $\nu$ corresponding to the Lebesgue uniform measure on $\fov$. We decided to use a small value of $N$ uniquely to speed up the computation of the MAEs. The other parameters (output dimension, $S_{\max{}}$, etc.) are identical to those of the previous tests and can be read, for example, in \Cref{tab:parameters nu test}.

In \Cref{fig:p and rho MAE} the MAEs obtained using several values of $\rho$ ($\rho = 2$, $10$, $100$, $1000$) and $p$ ($p = 0.1$, $0.25$, $0.5$, $0.75$, $1$, $1.5$, $2$) are shown. We can see that, as $p$ decreases or $\rho$ increases, the error profile becomes sharper and sharper around the ROI (look, for example, at how the MAE for $p = 1$ changes as $\rho$ gets larger). However, we can also see that, if $p$ is very small and $\rho$ is high, the accuracy of the VO kernel seems somewhat worse. For example, the MAE obtained using $p = 0.1$ is larger than that obtained using $p = 0.5$ when $\rho = 1000$. The cause of this phenomenon is likely to be the ill-conditioning of the gradient of $\lambda_{p, \rho, \eta}$, which deteriorates as $p$ gets smaller and $\rho$ larger.    
\begin{figure}[h]
    \centering
    \begin{subfigure}{0.49\textwidth}
        \includegraphics[width=\textwidth]{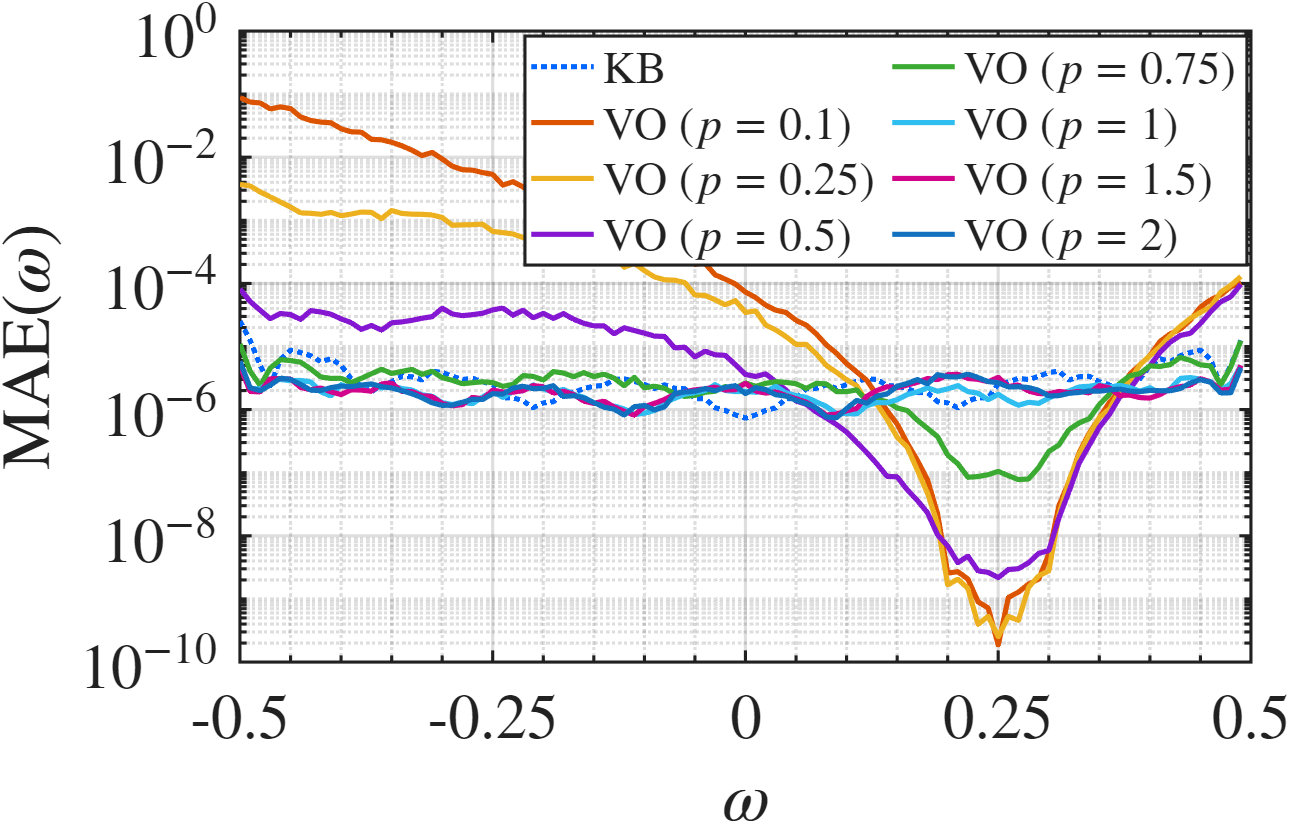}
    \end{subfigure}
    \begin{subfigure}{0.49\textwidth}
        \includegraphics[width=\textwidth]{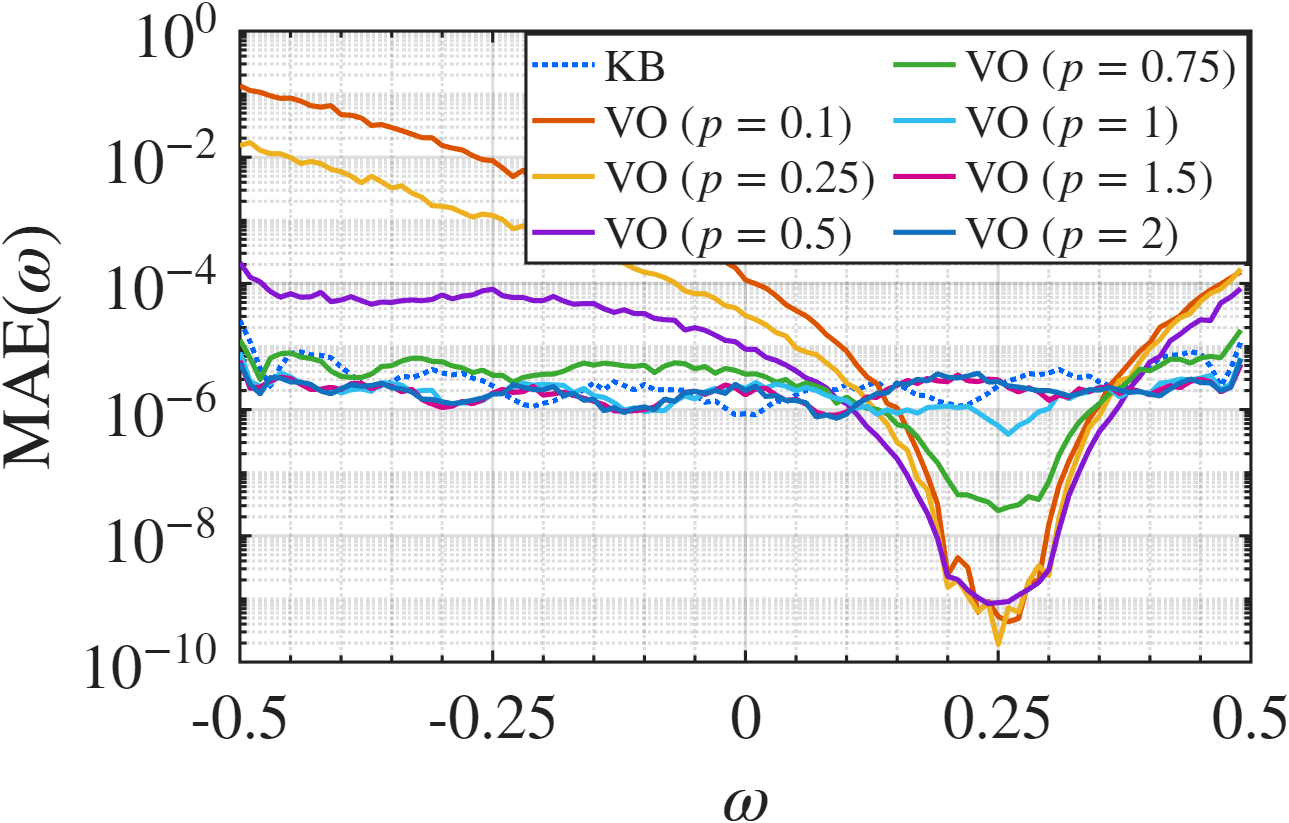}
    \end{subfigure}
    \begin{subfigure}{0.49\textwidth}
        \includegraphics[width=\textwidth]{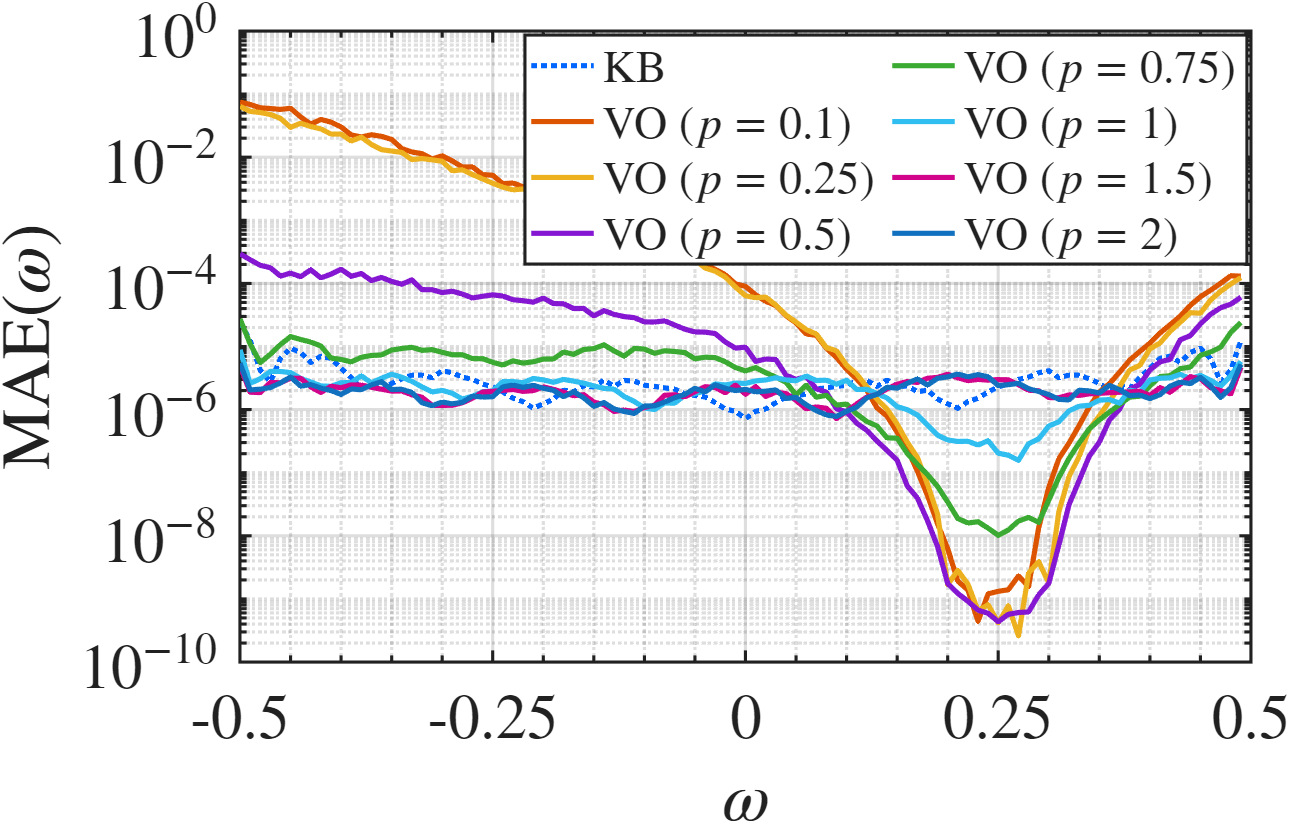}
    \end{subfigure}
    \begin{subfigure}{0.49\textwidth}
        \includegraphics[width=\textwidth]{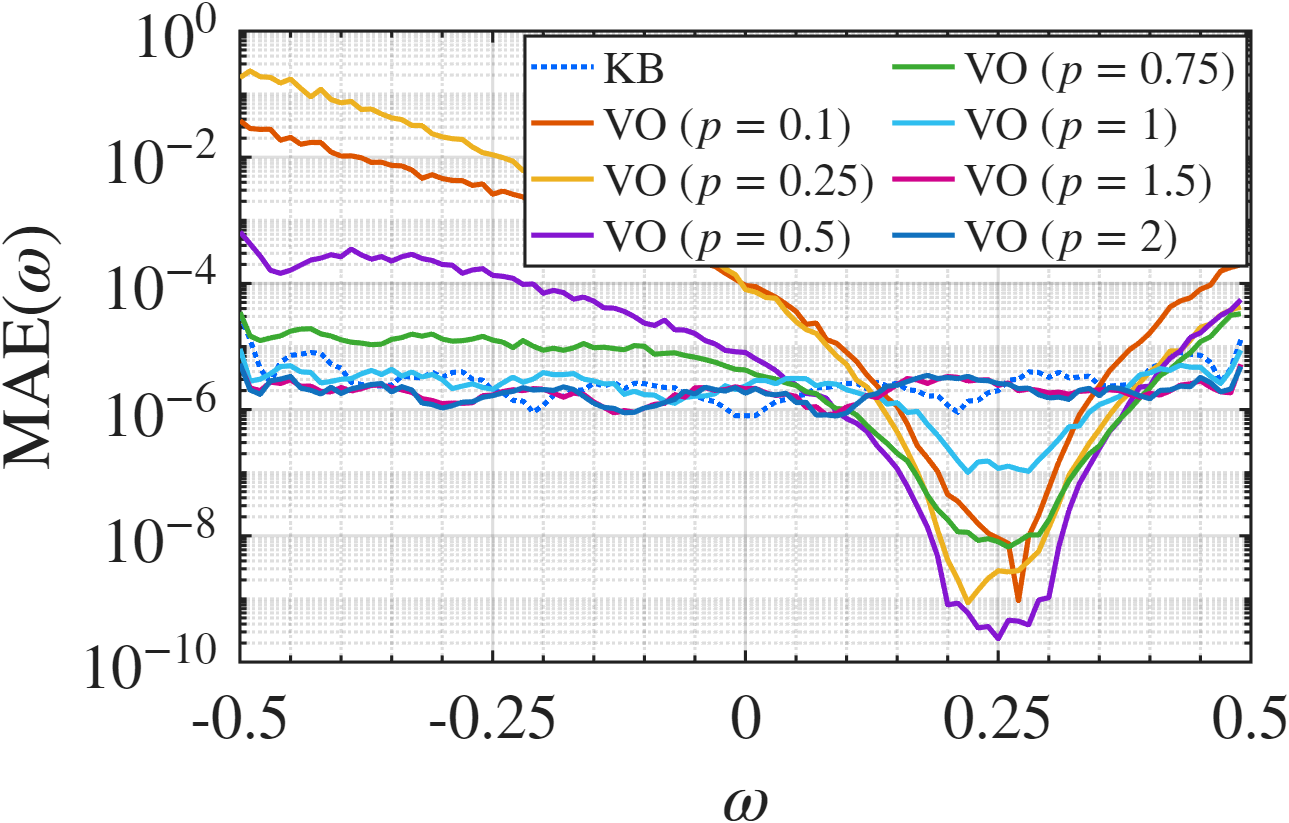}
    \end{subfigure}
    \caption{MAEs for varying $p$ and $\rho = 2$ (upper left corner), $\rho = 10$ (upper right corner), $\rho = 100$ (lower left corner), and $\rho = 1000$ (lower right corner). It is clearly visible that decreasing the value of $p$ generally prioritize the minimization of the error on the ROI, as wells as increasing $\rho$.}
    \label{fig:p and rho MAE}
\end{figure}

\section{Conclusions}
\label{ch:Conclusions}
In this paper we studied the design of gridding kernels and correcting functions for the adjoint NUFFT, with the goal of adapting them to the statistical distribution of the input samples and to the accuracy requirements of a specific application, rather than relying on a generic worst-case criterion. Assuming that the non-uniform time samples are asymptotically equidistributed with respect to a limit measure $\nu$ (\Cref{hyp:measure condition on time samples - informal}, formalized as \Cref{hyp:time samples are equidistributed - formal}), we derived in \Cref{ch:Error bound for the gridding algorithm} a new closed-form error bound $\ell_\nu(\omega, C)$ (\Cref{thm:error bound}) that generalizes the error bound of \cite{ye_optimal_2019}, obtained under a specific heuristic sampling-uniformity condition, to the case of an essentially arbitrary equidistribution measure $\nu$. We further showed, through the examples of analog-to-digital-converter jitter and swept-source OCT acquisition, that this equidistribution hypothesis is satisfied in relevant real-world sampling schemes.

Building on this bound, in \Cref{ch:Applying vector optimization to ell(x)} we recast the  selection of an optimal gridding kernel as a vector optimization (VO) problem, treating the error profile $\omega \mapsto \ell_\nu(\omega, C)$ as an element of an $L^p$ space ordered by its natural pointwise cone. This perspective generalizes the implicit optimization strategy underlying \cite{ye_optimal_2019} by removing its restriction to symmetric or otherwise prescribed regions of interest. We proved that the associated error profile operator $\Lambda_{p,\nu}$ is continuous, and established the existence of Pareto-optimal kernels within the finite-dimensional space spanned by the first prolate spheroidal wave functions (after smoothing), both in the unconstrained case (\Cref{thm:solution of vo problem}) and when a target error profile $\eta$ encoding a region of interest (ROI) is prescribed (\Cref{hyp:ROI hypothesis - formal}) and pursued via a penalty scalarizing functional $\lambda_{p,\rho,\eta}$ (\Cref{thm:Solution of the SAF problem}).

The numerical experiments of \Cref{ch:Numerical results}, based on a practical Legendre-polynomial parametrization of the kernel and covering signal sizes from $N = 10^3$ to $N = 10^8$, support this theoretical picture. For all target error profiles tested, the VO kernel achieved markedly better accuracy on the prescribed ROI than a Kaiser-Bessel (KB) kernel with same the support width $W$, while matching the accuracy of a wider KB kernel at a lower computational cost. In a multi-threaded environment, the VO kernel also generally outperformed the state-of-the-art FINUFFT library, although in single-threaded mode the faster method depended on $W$. A separate set of experiments showed that, when the sampling measure $\nu$ departs from the uniform (Lebesgue) case, the VO kernel improves global accuracy over the KB kernel at no additional computational cost, and the tests of \Cref{ch:Error profiles for varying p and rho} confirmed the proposed heuristics for tuning $p$ and $\rho$. We note that our implementation was found to be faster than FINUFFT even in the baseline, globally-optimal setting, for reasons we did not investigate and that may depend on hardware or compilation choices rather than on the gridding kernel itself. The timing comparisons above should be read with this caveat in mind.

The proposed approach has some limitations. Unlike FINUFFT, it does not allow the target accuracy to be fixed \textit{a priori}: for a given kernel support $W$, the resulting accuracy can only be assessed \textit{a posteriori}, requiring a trial-and-error search over $W$, even though each individual optimization converges quickly. All the configurations we tested were, moreover, limited by a round-off error floor of about $10^{-13}$, a limitation also shared by
FINUFFT. Finally, for $p \le 1$ the penalty scalarizing functional $\lambda_{p,\rho,\eta}$ is not differentiable, so no formal guarantee of Pareto optimality can be given for the numerically computed kernels in this regime, notwithstanding the consistent accuracy improvements observed
experimentally.

The framework developed here presupposes knowledge of the sampling measure $\nu$ and, when \Cref{hyp:ROI hypothesis - formal} is invoked, of a target error profile $\eta$, both of which must generally be supplied by a domain expert (for example using a human-in-the-loop methodology). Future work shall explore data-driven procedures to accomplish this task, for instance based on neural networks.

Finally, it is worth mentioning that \Cref{hyp:ROI hypothesis - formal} is likely applicable also to the forward NUFFT.

\section*{Acknowledgements}
We thank Prof. Jeffrey Fessler of the University of Michigan for his advice, and we also thank Prof. Giuseppe Molteni and Prof. Marta Calanchi of Università degli Studi di Milano for useful conversation.

\null

\noindent This work was financed by \textit{NextGenerationEU}, Mission 4, Component 1, CUP G43C23001640006.

\FloatBarrier
\printbibliography

@book{miettinen_nonlinear_1998,
	address = {Boston, MA},
	series = {International {Series} in {Operations} {Research} \& {Management} {Science}},
	title = {Nonlinear {Multiobjective} {Optimization}},
	volume = {12},
	copyright = {http://www.springer.com/tdm},
	isbn = {978-1-4613-7544-9},
	publisher = {Springer US},
	author = {Miettinen, Kaisa},
	editor = {Hillier, Frederick S.},
	year = {1998},
	doi = {10.1007/978-1-4615-5563-6},
}

@article{wierzbick_basic_1977,
    series = {Series {Optimization}},
    title = {Basic properties of scalarizing functionals for multiobjective optimization},
    volume = {8},
    copyright = {Copyright Taylor and Francis Group, LLC},
    doi = {10.1080/02331937708842405},
    number = {1},
    journal = {Mathematische Operationsforschung und Statistik.},
    author = {Wierzbicki, Andrzej P.},
    month = jan,
    year = {1977},
    note = {Publisher: Akademic-Verlag},
}

@book{jahn_vector_2011,
    address = {Berlin, Heidelberg},
    title = {Vector {Optimization}: {Theory}, {Applications}, and {Extensions}},
    copyright = {https://www.springernature.com/gp/researchers/text-and-data-mining},
    isbn = {978-3-642-17004-1},
    shorttitle = {Vector {Optimization}},
    publisher = {Springer},
    author = {Jahn, Johannes},
    year = {2011},
    doi = {10.1007/978-3-642-17005-8},
}

@article{jahn_unified_2023,
    title = {A unified approach to {Bishop}-{Phelps} and scalarizing functionals},
    volume = {5},
    doi = {10.23952/jano.5.2023.1.02},
    number = {1},
    urldate = {2025-04-23},
    journal = {Journal of Applied and Numerical Optimization},
    author = {Jahn, Johannes},
    month = apr,
    year = {2023},
    pages = {5--25},
}

@article{ye_optimal_2019,
    title = {Optimal gridding and degridding in radio interferometry imaging},
    issn = {0035-8711, 1365-2966},
    doi = {10.1093/mnras/stz2970},
    journal = {Monthly Notices of the Royal Astronomical Society},
    author = {Ye, Haoyang and Gull, Stephen F. and Tan, Sze M. and Nikolic, Bojan},
    month = nov,
    year = {2019},
}

@book{ansorge_physics_2016,
    title = {The {Physics} and {Mathematics} of {MRI}},
    isbn = {978-1-68174-068-3},
    publisher = {Morgan \& Claypool Publishers},
    author = {Ansorge, Richard and Graves, Martin},
    month = oct,
    year = {2016},
}

@article{yang_mean_2014,
    title = {Mean square optimal {NUFFT} approximation for efficient non-{Cartesian} {MRI} reconstruction},
    volume = {242},
    doi = {10.1016/j.jmr.2014.01.016},
    journal = {Journal of Magnetic Resonance (San Diego, Calif.: 1997)},
    author = {Yang, Zhili and Jacob, Mathews},
    month = may,
    year = {2014},
    pages = {126--135},
}

@article{Fessler_minmaxNUFFT,
    title = {Nonuniform {F}ast {Fourier} transforms using min-max interpolation},
    volume = {51},
    doi = {10.1109/TSP.2002.807005},
    number = {2},
    journal = {IEEE Transactions on Signal Processing},
    author = {Fessler, J.A. and Sutton, B.P.},
    month = feb,
    year = {2003},
    pages = {560--574},
}

@article{schwab_optimal_1980,
    title = {Optimal {Gridding}},
    volume = {National Radio Astronomy Observatory},
    number = {132},
    journal = {VLA Scientific Memoranda},
    author = {Schwab, Frederic R.},
    month = oct,
    year = {1980},
}

@incollection{brouw_aperture_1975,
    series = {Radio {Astronomy}},
    title = {Aperture {Synthesis}},
    volume = {14},
    booktitle = {Methods in {Computational} {Physics}: {Advances} in {Research} and {Applications}},
    publisher = {Elsevier},
    author = {Brouw, W. N.},
    editor = {Alder, BERNI and Fernbach, SIDNEY and Rotenberg, MANUEL},
    month = jan,
    year = {1975},
    doi = {10.1016/B978-0-12-460814-6.50008-5},
    pages = {131--175},
}

@article{osullivan_fast_1985,
    title = {A {Fast} {Sinc} {Function} {Gridding} {Algorithm} for {Fourier} {Inversion} in {Computer} {Tomography}},
    volume = {4},
    doi = {10.1109/TMI.1985.4307723},
    number = {4},
    urldate = {2025-01-28},
    journal = {IEEE Transactions on Medical Imaging},
    author = {O'Sullivan, J. D.},
    month = dec,
    year = {1985},
    pages = {200--207},
}

@article{lewitt_reconstruction_1983,
    title = {Reconstruction algorithms: {Transform} methods},
    volume = {71},
    doi = {10.1109/PROC.1983.12597},
    number = {3},
    urldate = {2025-04-24},
    journal = {Proceedings of the IEEE},
    author = {Lewitt, R.M.},
    month = mar,
    year = {1983},
    pages = {390--408},
}

@book{Brigham_FFT,
    edition = {First},
    title = {The fast {Fourier} transform and its applications},
    isbn = {978-0-13-307505-2},
    publisher = {Pearson},
    author = {Oran Brigham, E.},
    month = apr,
}

@article{cooley_algorithm_1965,
    title = {An algorithm for the machine calculation of complex {Fourier} series},
    volume = {19},
    issn = {0025-5718, 1088-6842},
    doi = {10.1090/S0025-5718-1965-0178586-1},
    number = {90},
    journal = {Mathematics of Computation},
    author = {Cooley, James W. and Tukey, John W.},
    year = {1965},
    pages = {297--301},
}

@article{Dutt_NUFFT,
    title = {Fast {Fourier} {Transforms} for {Nonequispaced} {Data}},
    volume = {14},
    doi = {10.1137/0914081},
    number = {6},
    journal = {SIAM Journal on Scientific Computing},
    author = {Dutt, A. and Rokhlin, V.},
    month = nov,
    year = {1993},
    pages = {1368--1393},
}

@article{jackson_selection_1991,
    title = {Selection of a convolution function for {Fourier} inversion using gridding (computerised tomography application)},
    volume = {10},
    doi = {10.1109/42.97598},
    number = {3},
    urldate = {2025-01-28},
    journal = {IEEE Transactions on Medical Imaging},
    author = {Jackson, J.I. and Meyer, C.H. and Nishimura, D.G. and Macovski, A.},
    month = sep,
    year = {1991},
    pages = {473--478},
}

@article{schomberg_gridding_1995,
    title = {The gridding method for image reconstruction by {Fourier} transformation},
    volume = {14},
    doi = {10.1109/42.414625},
    number = {3},
    urldate = {2025-01-24},
    journal = {IEEE Transactions on Medical Imaging},
    author = {Schomberg, H. and Timmer, J.},
    month = sep,
    year = {1995},
    pages = {596--607},
}

@article{slepian_prolate_1961,
    title = {Prolate spheroidal wave functions, fourier analysis and uncertainty - {I}},
    volume = {40},
    doi = {10.1002/j.1538-7305.1961.tb03976.x},
    number = {1},
    urldate = {2025-01-31},
    journal = {The Bell System Technical Journal},
    author = {Slepian, D. and Pollak, H. O.},
    month = jan,
    year = {1961},
    pages = {43--63},
}

@article{landau_prolate_1962,
    title = {Prolate spheroidal wave functions, fourier analysis and uncertainty — {III}: {The} dimension of the space of essentially time- and band-limited signals},
    volume = {41},
    doi = {10.1002/j.1538-7305.1962.tb03279.x},
    number = {4},
    journal = {The Bell System Technical Journal},
    author = {Landau, H. J. and Pollak, H. O.},
    month = jul,
    year = {1962},
    pages = {1295--1336},
}

@article{landau_prolate_1961,
    title = {Prolate spheroidal wave functions, fourier analysis and uncertainty — {II}},
    volume = {40},
    doi = {10.1002/j.1538-7305.1961.tb03977.x},
    number = {1},
    urldate = {2025-02-04},
    journal = {The Bell System Technical Journal},
    author = {Landau, H. J. and Pollak, H. O.},
    month = jan,
    year = {1961},
    pages = {65--84},
}

@article{slepian_prolate_1978,
    title = {Prolate spheroidal wave functions, fourier analysis, and uncertainty — {V}: the discrete case},
    volume = {57},
    doi = {10.1002/j.1538-7305.1978.tb02104.x},
    number = {5},
    urldate = {2025-02-04},
    journal = {The Bell System Technical Journal},
    author = {Slepian, D.},
    month = may,
    year = {1978},
    pages = {1371--1430},
}

@book{luc_theory_1989,
    address = {Berlin, Heidelberg},
    series = {Lecture {Notes} in {Economics} and {Mathematical} {Systems}},
    title = {Theory of {Vector} {Optimization}},
    volume = {319},
    isbn = {978-3-540-50541-9},
    publisher = {Springer},
    author = {Luc, Dinh The},
    editor = {Beckmann, M. and Krelle, W.},
    year = {1989},
    doi = {10.1007/978-3-642-50280-4},
}

@article{cao_locally_1995,
    title = {Locally focused {MRI}},
    volume = {34},
    doi = {10.1002/mrm.1910340611},
    number = {6},
    urldate = {2025-07-14},
    journal = {Magnetic Resonance in Medicine},
    author = {Cao, Yue and Levin, David N. and Yao, Lian},
    month = dec,
    year = {1995},
    pages = {858--867},
}

@article{meignan_review_2015,
    title = {A {Review} and {Taxonomy} of {Interactive} {Optimization} {Methods} in {Operations} {Research}},
    volume = {5},
    doi = {10.1145/2808234},
    number = {3},
    journal = {ACM Trans. Interact. Intell. Syst.},
    author = {Meignan, David and Knust, Sigrid and Frayret, Jean-Marc and Pesant, Gilles and Gaud, Nicolas},
    year = {2015},
    pages = {17:1--17:43},
}

@article{Nguyen_NUFFT,
    title = {The {Regular} {Fourier} {Matrices} and {Nonuniform} {Fast} {Fourier} {Transforms}},
    volume = {21},
    doi = {10.1137/S1064827597325712},
    number = {1},
    urldate = {2024-02-07},
    journal = {SIAM Journal on Scientific Computing},
    author = {Nguyen, Nhu and Liu, Qing Huo},
    month = jan,
    year = {1999},
    pages = {283--293},
}

@article{li_compensation_2023,
    title = {A {Compensation} {Method} for {Nonlinearity} {Errors} in {Optical} {Interferometry}},
    volume = {23},
    copyright = {http://creativecommons.org/licenses/by/3.0/},
    doi = {10.3390/s23187942},
    number = {18},
    urldate = {2024-09-06},
    journal = {Sensors},
    author = {Li, Yanlu and Dieussaert, Emiel},
    month = jan,
    year = {2023},
    pages = {7942},
}

@article{hogg_synthesis_1969,
    title = {Synthesis of {Brightness} {Distribution} in {Radio} {Sources}},
    volume = {74},
    doi = {10.1086/110924},
    journal = {The Astronomical Journal},
    author = {Hogg, D. E. and MacDonald, G. H. and Conway, R. G. and Wade, C. M.},
    month = dec,
    year = {1969},
    pages = {1206},
}

@article{gallagher_introduction_2008,
    title = {An introduction to the {Fourier} transform: relationship to {MRI}},
    volume = {190},
    doi = {10.2214/AJR.07.2874},
    number = {5},
    journal = {AJR. American journal of roentgenology},
    author = {Gallagher, Thomas A. and Nemeth, Alexander J. and Hacein-Bey, Lotfi},
    month = may,
    year = {2008},
    pages = {1396--1405},
}

@book{varian_intermediate_2024,
    title = {Intermediate {Microeconomics} with {Calculus} : {A} {Modern} {Approach}},
    isbn = {978-1-324-03443-8},
    publisher = {W.W. Norton \& Co.},
    author = {Varian, Hal R.},
    month = jan,
    year = {2024},
}

@article{fessler_nufft-based_2007,
    title = {On {NUFFT}-based gridding for non-{Cartesian} {MRI}},
    volume = {188},
    doi = {10.1016/j.jmr.2007.06.012},
    number = {2},
    journal = {Journal of Magnetic Resonance},
    author = {Fessler, Jeffrey A.},
    month = oct,
    year = {2007},
    pages = {191--195},
}

@article{barnett_parallel_2019,
    title = {A {Parallel} {Nonuniform} {Fast} {Fourier} {Transform} {Library} {Based} on an “{Exponential} of {Semicircle}" {Kernel}},
    volume = {41},
    doi = {10.1137/18M120885X},
    number = {5},
    urldate = {2025-10-16},
    journal = {SIAM Journal on Scientific Computing},
    author = {Barnett, Alexander H. and Magland, Jeremy and Af Klinteberg, Ludvig},
    month = jan,
    year = {2019},
    pages = {C479--C504},
}

@article{wagdy_effect_1987,
    title = {Effect of {ADC} quantization errors on some periodic signal measurements},
    volume = {IM-36},
    doi = {10.1109/TIM.1987.6312595},
    number = {4},
    journal = {IEEE Transactions on Instrumentation and Measurement},
    author = {Wagdy, Mahmoud Fawzy},
    month = dec,
    year = {1987},
    pages = {983--989},
}

@article{gans_measurement_1983,
    title = {The {Measurement} and {Deconvolution} of {Time} {Jitter} in {Equivalent}-{Time} {Waveform} {Samplers}},
    volume = {32},
    doi = {10.1109/TIM.1983.4315022},
    number = {1},
    journal = {IEEE Transactions on Instrumentation and Measurement},
    author = {Gans, William L.},
    month = mar,
    year = {1983},
    pages = {126--133},
}

@article{wagdy_effect_1990,
    title = {Effect of sampling jitter on some sine wave measurements},
    volume = {39},
    doi = {10.1109/19.50422},
    number = {1},
    journal = {IEEE Transactions on Instrumentation and Measurement},
    author = {Wagdy, M.F. and Awad, S.S.},
    month = feb,
    year = {1990},
    pages = {86--89},
}

@article{rivet_complex_2016,
    title = {Complex master slave interferometry},
    volume = {24},
    copyright = {https://doi.org/10.1364/OA\_License\_v1\#VOR-OA},
    doi = {10.1364/OE.24.002885},
    number = {3},
    urldate = {2025-10-16},
    journal = {Optics Express},
    author = {Rivet, Sylvain and Maria, Michael and Bradu, Adrian and Feuchter, Thomas and Leick, Lasse and Podoleanu, Adrian},
    month = feb,
    year = {2016},
    pages = {2885},
}

@book{billingsley_convergence_1999,
    address = {New York},
    edition = {2. ed},
    series = {Probability and statistics section},
    title = {Convergence of probability measures},
    isbn = {978-0-471-19745-4},
    publisher = {Wiley},
    author = {Billingsley, Patrick},
    year = {1999},
}

@article{layton_single_2013,
    title = {Single shot trajectory design for region-specific imaging using linear and nonlinear magnetic encoding fields},
    volume = {70},
    doi = {10.1002/mrm.24494},
    number = {3},
    journal = {Magnetic Resonance in Medicine},
    author = {Layton, Kelvin J. and Gallichan, Daniel and Testud, Frederik and Cocosco, Chris A. and Welz, Anna M. and Barmet, Christoph and Pruessmann, Klaas P. and Hennig, Jürgen and Zaitsev, Maxim},
    month = sep,
    year = {2013},
    pages = {684--696},
}

@article{schmid_nonequispaced_2016,
    title = {Nonequispaced grid sampling in photoacoustics with a nonuniform fast {Fourier} transform},
    volume = {21},
    doi = {10.1117/1.JBO.21.1.015005},
    number = {1},
    journal = {Journal of Biomedical Optics},
    publisher = {SPIE},
    author = {Schmid, Julian and Glatz, Thomas and Zabihian, Behrooz and Liu, Mengyang and Drexler, Wolfgang and Scherzer, Otmar},
    month = jan,
    year = {2016},
    pages = {015005},
}

@article{liu_joint_2022,
    title = {Joint disease classification and lesion segmentation via one-stage attention-based convolutional neural network in {OCT} images},
    volume = {71},
    doi = {10.1016/j.bspc.2021.103087},
    journal = {Biomedical Signal Processing and Control},
    author = {Liu, Xiaoming and Bai, Yingjie and Cao, Jun and Yao, Junping and Zhang, Ying and Wang, Man},
    month = jan,
    year = {2022},
    pages = {103087},
}

@article{greengard_accelerating_2004,
    title = {Accelerating the {Nonuniform} {Fast} {Fourier} {Transform}},
    volume = {46},
    doi = {10.1137/S003614450343200X},
    number = {3},
    urldate = {2026-09-01},
    journal = {SIAM Review},
    author = {Greengard, Leslie and Lee, June-Yub},
    month = jan,
    year = {2004},
    pages = {443--454},
}

@misc{keiner_nfft_nodate,
    title = {{NFFT}},
    url = {https://www-user.tu-chemnitz.de/~potts/nfft/},
    author = {Keiner, Jens and Kunis, Stefan and Potts, Daniel},
    note = {Available online at: \url{https://www-user.tu-chemnitz.de/~potts/nfft/}, 
           last accessed on 01.09.2026}
}

@misc{flatiron_institute_improved_nodate,
    title = {Improved spreading kernel, tolerance matching and testing, for v2.5.0 · {Discussion} \#798},
    url = {https://github.com/flatironinstitute/finufft/discussions/798},
    urldate = {2026-09-10},
    journal = {GitHub},
    author = {{Flatiron Institute}},
    note = {Available online at: \url{https://github.com/flatironinstitute/finufft/discussions/798}, last accessed on 10.09.2026}
}

@article{hale_algorithm_2014,
    title = {An {Algorithm} for the {Convolution} of {Legendre} {Series}},
    volume = {36},
    issn = {1064-8275, 1095-7197},
    url = {https://epubs.siam.org/doi/10.1137/140955835},
    doi = {10.1137/140955835},
    language = {en},
    number = {3},
    urldate = {2026-09-10},
    journal = {SIAM Journal on Scientific Computing},
    author = {Hale, Nicholas and Townsend, Alex},
    month = jan,
    year = {2014},
    pages = {A1207--A1220},
}

\end{document}